\pdfoutput=1
\documentclass[11pt]{article}
\usepackage{authblk}

\usepackage{color}
\usepackage{helvet}         
\usepackage{courier}        
\usepackage{type1cm}        
\usepackage{makeidx}         
\usepackage{graphicx}        
\usepackage{multicol}        
\usepackage[bottom]{footmisc}
\usepackage{amsmath}
\usepackage{amssymb}
\usepackage{bbold}
\usepackage{amsthm}
\usepackage{subcaption}
\usepackage{comment}

\usepackage{wasysym} 
\usepackage{color}
\usepackage{xcolor}

\usepackage[final]{hyperref}   
\hypersetup{
    linktoc=page,
    linkcolor=red,          
    citecolor=blue,        
    filecolor=blue,      
    urlcolor=cyan,
   colorlinks=true           
}

\newtheorem{theorem}{Theorem}
\newtheorem{corollary}[theorem]{Corollary}

\newtheorem{lemma}[theorem]{Lemma}

\newtheorem{proposition}[theorem]{Proposition}

\newtheorem{definition}[theorem]{Definition}
\newtheorem{remark}[theorem]{Remark}

\usepackage[top=2.8cm, bottom=3cm, left=3cm, right=3cm]{geometry}
\numberwithin{theorem}{section}
\numberwithin{figure}{section}
\numberwithin{equation}{section}

\DeclareMathOperator{\cov}{cov}
\DeclareMathOperator{\dist}{dist}
\DeclareMathOperator{\SLE}{SLE}
\DeclareMathOperator{\CLE}{CLE}

\DeclareMathOperator{\ext}{ext}

\DeclareMathOperator{\hcap}{hcap}
\DeclareMathOperator{\diam}{diam}

\DeclareMathOperator{\simple}{simple}
\DeclareMathOperator{\Leb}{Leb}

\def\rm{\reversemarginpar}

\title{\bf On the convergence of FK-Ising percolation to $\SLE(16/3, 16/3-6)$}
{\small\author[1]{Christophe Garban\thanks{garban@math.univ-lyon1.fr. C.G. is supported by the 
ANR grant \textsc{Liouville} ANR-15-CE40-0013 and the ERC grant LiKo 676999.}}
\author[2]{Hao Wu\thanks{hao.wu.proba@gmail.com. H.W. is supported by the Thousand Talents Plan for Young Professionals. }}
\affil[1]{Institut Camille Jordan, Universit\'{e} Claude Bernard Lyon 1, France}
\affil[2]{Yau Mathematical Sciences Center, Tsinghua University, China}}

\date{}

\begin{document}

\newcommand{\eps}{\epsilon}
\newcommand{\ov}{\overline}
\newcommand{\U}{\mathbb{U}}
\newcommand{\T}{\mathbb{T}}
\newcommand{\HH}{\mathbb{H}}
\newcommand{\LA}{\mathcal{A}}
\newcommand{\LB}{\mathcal{B}}
\newcommand{\LC}{\mathcal{C}}
\newcommand{\LD}{\mathcal{D}}
\newcommand{\LF}{\mathcal{F}}
\newcommand{\LK}{\mathcal{K}}
\newcommand{\LE}{\mathcal{E}}
\newcommand{\LG}{\mathcal{G}}
\newcommand{\LL}{\mathcal{L}}
\newcommand{\LM}{\mathcal{M}}
\newcommand{\LQ}{\mathcal{Q}}
\newcommand{\LU}{\mathcal{U}}
\newcommand{\LV}{\mathcal{V}}
\newcommand{\LZ}{\mathcal{Z}}
\newcommand{\LH}{\mathcal{H}}
\newcommand{\R}{\mathbb{R}}
\newcommand{\C}{\mathbb{C}}
\newcommand{\N}{\mathbb{N}}
\newcommand{\Z}{\mathbb{Z}}
\newcommand{\E}{\mathbb{E}}
\newcommand{\PP}{\mathbb{P}}
\newcommand{\QQ}{\mathbb{Q}}
\newcommand{\A}{\mathbb{A}}
\newcommand{\one}{\mathbb{1}}
\newcommand{\bn}{\mathbf{n}}
\newcommand{\MR}{MR}
\newcommand{\cond}{\,|\,}
\newcommand{\la}{\langle}
\newcommand{\ra}{\rangle}
\newcommand{\tree}{\Upsilon}

\newcommand{\D}{\mathbb{D}}
\renewcommand{\H}{\mathbb{H}}
\newcommand{\Q}{\mathbb{Q}}
\renewcommand{\S}{\mathbb{S}}

\def\calA{\mathcal{A}}
\def\calB{\mathcal{B}}
\def\calC{\mathcal{C}}
\def\calD{\mathcal{D}}
\def\calE{\mathcal{E}}
\def\calF{\mathcal{F}}
\def\calG{\mathcal{G}}
\def\calL{\mathcal{L}}
\def\calH{\mathcal{H}}
\def\calI{\mathcal{I}}
\def\calN{\mathcal{N}}
\def\calS{\mathcal{S}}
\def\calT{\mathcal{T}}
\def\calW{\mathcal{W}}

\def\g{\mathrm{good}}
\def\ins{\mathrm{in}}
\def\ext{\mathrm{ext}}
\def\SLE{\mathrm{SLE}}
\def\CLE{\mathrm{CLE}}
\def\diam{\mathrm{diam}}
\def\rad{\mathop{\mathrm{rad}}}
\def\var{\mathop{\mathrm{Var}}}
\def\Var#1{\mathrm{Var}\bigl[ #1\bigr]}
\def\cov{\mathrm{Cov}}
\def\Cov#1{\mathrm{Cov}\bigl[ #1\bigr]}
\def\sign{\mathrm{sign}}
\def\dist{\mathrm{dist}}
\def\length{\mathop{\mathrm{length}}}
\def\area{\mathop{\mathrm{area}}}
\def\ceil#1{\lceil{#1}\rceil}
\def\floor#1{\lfloor{#1}\rfloor}
\def\Im{{\rm Im}\,}
\def\Re{{\rm Re}\,}

\def\de{\delta}

\def\Haus{\mathfrak H}

\def\P{\mathbb{P}} 
\def\E{\mathbb{E}} 
\def\md{\mid}

\def\Bb#1#2{{\def\md{\bigm| }#1\bigl[#2\bigr]}}
\def\BB#1#2{{\def\md{\Bigm| }#1\Bigl[#2\Bigr]}}

\def\Pb{\Bb\P}
\def\Eb{\Bb\E}
\def\PB{\BB\P}
\def\EB{\BB\E}
\def\FK#1#2#3{{\def\md{\bigm| } \P_{#1}^{\,#2}  \bigl[  #3 \bigr]}}
\def\EFK#1#2#3{{\def\md{\bigm| } \E_{#1}^{\,#2}  \bigl[  #3 \bigr]}}

\def \p {{\partial}}
\def\closure{\overline}
\def\ev#1{{\mathcal{#1}}}

\def\<#1{\langle #1\rangle}
\newcommand{\red}[1]{\textcolor{red}{#1}}
\newcommand{\blue}[1]{\textcolor{blue}{#1}}
\newcommand{\purple}[1]{\textcolor{purple}{#1}}

\def\nn{\nonumber}
\def\bi{\begin{itemize}}  
\def\ei{\end{itemize}}
\def\bnum{\begin{enumerate}} 
\def\enum{\end{enumerate}}
\def\ni{\noindent}
\def\bf{\bfseries}
\let\sc\scshape

\maketitle

\begin{abstract}

We give a simplified and complete proof of the convergence of the chordal exploration process in critical FK-Ising percolation to chordal $\SLE_\kappa( \kappa-6)$ with $\kappa=16/3$. Our proof follows the classical excursion-construction of $\SLE_\kappa(\kappa-6)$ processes in the continuum and we are thus led to introduce suitable cut-off  stopping times in order to analyse the behaviour of the driving function of the discrete system when Dobrushin boundary condition collapses to a single point.
Our proof is very different from \cite{KemppainenSmirnovBoundaryTouchingLoopsFKIsing, KemppainenSmirnovFullLimitFKIsing} as it only relies on the convergence to the chordal $\SLE_{\kappa}$ process in Dobrushin boundary condition and does not require the introduction of a new observable. Still, it relies crucially on several ingredients:
\bi
\item[a)] the powerful topological framework developed in \cite{KemppainenSmirnovRandomCurves} as well as its follow-up paper \cite{CDCHKSConvergenceIsingSLE},
\item[b)] the strong RSW Theorem from \cite{ChelkakDuminilHonglerCrossingprobaFKIsing},
\item[c)] the proof is inspired from the appendix A in \cite{BenoistHonglerIsingCLE}.
\ei

One important emphasis of this paper is to carefully write down some properties which are often considered {\em folklore} in the literature but which are only justified so far by hand-waving arguments. The main examples of these are:
\bi
\item[1)] the convergence of natural discrete stopping times to their continuous analogues. (The usual hand-waving argument destroys the spatial Markov property). 
\item[2)] the fact that the discrete spatial Markov property is preserved in the the scaling limit. (The enemy being that $\Eb{X_n \md Y_n}$ does not necessarily converge to $\Eb{X\md Y}$ when $(X_n,Y_n)\to (X,Y)$).
\ei 
We end the paper with a detailed sketch of the convergence to radial  $\SLE_\kappa( \kappa-6)$ when $\kappa=16/3$ as well as the derivation of Onsager's one-arm exponent $1/8$.
\end{abstract}
\newpage

\section{Introduction}

The random cluster model on a finite graph $\Omega=(V(\Omega), E(\Omega))\subset \Z^2$ is a probability measure on bond configurations $\omega=(\omega_e: e\in E(\Omega))\in \{0,1\}^{E(\Omega)}$: 
\[\phi_{p,q,\Omega}[\omega]\propto p^{o(\omega)}(1-p)^{c(\omega)}q^{k(\omega)},\]
where $o(\omega)$ (resp. $c(\omega)$) denotes the number of open edges (resp. closed edges) in $\omega$ and $k(\omega)$ denotes the number of clusters in $\omega$. This model was introduced by Fortuin and Kasteleyn in 1969 and this model is closely related to the Ising model and the Potts model. When $q\ge 1$, the model enjoys FKG inequality which makes it possible to consider the infinite volume measures of the model. For $q\ge 1$, there exists a critical value $p_c$ for each $q$ such that, for $p>p_c$, any infinite volume measure has an infinite cluster; whereas, for $p<p_c$, any infinite volume measure has no infinite cluster. This dichotomy does not tell what happens at criticality $p=p_c$ and the critical phase is of great interest. 
The value of $p_c$ depends on $q$: $p_c(q)=\sqrt{q}/(1+\sqrt{q})$ for $q\ge 1$. This result was proved for Bernoulli percolation ($q=1$) by Kesten in 1980, and was derived for $q=2$ by Onsager in 1944 using the connection with Ising model. It was proved for $q\ge 1$ in \cite{BeffaraDuminilCopinSelfDual2DRCM}.
When $q\in [1,4]$, the critical phase is believed to be conformally invariant and the interface at criticality is conjectured to converge to $\SLE_{\kappa}$ where 
\[\kappa=4\pi/\arccos(-\sqrt{q}/2).\]

Conformal invariance is proved for $q=2$ in the celebrated works  \cite{Stas, ChelkakSmirnovIsing} while the convergence to $\SLE_\kappa$ is proved in  \cite{CDCHKSConvergenceIsingSLE}. 
When $q=2$, the random-cluster model is also called FK-Ising percolation. Precisely, the conclusion proved in \cite{Stas,ChelkakSmirnovIsing, CDCHKSConvergenceIsingSLE} is the following: consider the critical FK-Ising percolation on a simply connected domain $\Omega$ with Dobrushin boundary condition, the interface converges in law to $\SLE_{\kappa}$.  What about the convergence with other boundary conditions? The simplest boundary condition after the Dobrushin one is the fully wired boundary condition. We will give detailed description of the interface with fully wired boundary condition in Section~\ref{subsec::pre_fk}. 
The convergence of the interface with fully wired boundary condition is the main topic of this article. 

\begin{theorem}\label{thm::cvg_interface_chordal}
Let $(\Omega; a,b,c)$ be either a Jordan domain or the upper half plane $\H$ with three marked points $a, b, c$ on its boundary. 
Let $(\Omega^{\delta}; a^{\delta}, b^{\delta}, c^\delta)$ be a sequence of discrete domains on $\delta\Z^2$ converging to  $(\Omega; a, b, c )$ in the Carath\'{e}odory sense. Then, as $\delta\to 0$, the exploration path of the critical FK-Ising model in the domain $(\Omega^{\delta}; a^{\delta}, b^{\delta})$ with Dobrushin wired$/$free boundary condition and targeted at $c^\delta$,  converges weakly to the chordal $\SLE_{\kappa}(\kappa-6)$ from $a$ to $c$ with force point at $b$ and with $\kappa=16/3$. The case $a\equiv b$ corresponds to an exploration path in a fully wired (or fully free) domain. 
\end{theorem}

The same conclusion was also proved in \cite{KemppainenSmirnovBoundaryTouchingLoopsFKIsing}, but our proof is very different from the one there. In \cite{KemppainenSmirnovBoundaryTouchingLoopsFKIsing}, the authors constructed the so-called holomorphic observable for fully wired boundary condition which is a generalization of the observable constructed in \cite{ChelkakSmirnovIsing} for Dobrushin boundary condition; and then extract information from the observable to characterize the scaling limit. Our approach is different and it only relies on the convergence to the chordal $\SLE$ process and the powerful topological tool developed in \cite{KemppainenSmirnovRandomCurves}. This result also appears in \cite[Appendix~A]{BenoistHonglerIsingCLE} and our proof follows a similar strategy as in that appendix except full detail is given here.
\smallskip

In order to explain our approach, let us first describe the connection between $\SLE_{\kappa}(\kappa-6)$ and $\SLE_{\kappa}$. Fix $\kappa\in (4,8)$, the process $\SLE_{\kappa}(\kappa-6)$ is the Loewner chain (see Section~\ref{subsec::pre_chordalchain}) with the driving function $W$ which is the solution to the following SDE system:
\begin{equation}
\label{eqn::slekappaminussix_sde}
dW_t=\sqrt{\kappa}dB_t+\frac{(\kappa-6)dt}{W_t-V_t},\quad W_0=0;\quad dV_t=\frac{2dt}{V_t-W_t},\quad V_0=x\ge 0,
\end{equation}
where $B$ is a standard one-dimensional Brownian motion. 
The corresponding Loewner chain is called $\SLE_{\kappa}(\kappa-6)$ in $\HH$ from $0$ to $\infty$ with force point $x$. Set $\theta_t=V_t-W_t$, we find that $\theta_t/\sqrt{\kappa}$ is a Bessel process of dimension $3-8/\kappa$. Note that $\theta_t$ is the renormalized harmonic measure (see Section~\ref{s.RHM}) of the right side of $\eta[0,t]$ union $[0,x]$ seen from infinity.
 
As the process is invariant by scaling, one can define the process in any simply connected domain via conformal image. The process has the following special property---target-independence:
Suppose $(\Omega; a, b, c)$ is a simply connected domain with three distinct degenerate prime ends $a,b,c$ on the boundary in counterclockwise order. 
Then an $\SLE_{\kappa}(\kappa-6)$ in $\Omega$ from $a$ to $c$ with force point $b$, then, up to the disconnection time---the first hitting time of the boundary arc $\partial_{bc}$, it has the same law as an $\SLE_{\kappa}$ in $\Omega$ from $a$ to $b$, up to the disconnection time. 
This target-independent property allows us to decompose $\SLE_{\kappa}(\kappa-6)$ process into $\SLE_{\kappa}$ excursions as follows. 

Fix some cut-off $\eps>0$, define $T_1^{\eps}$ to be the first time that $\theta$ reaches $\eps$ and define $S_1^{\eps}$ to be the first time after $T_1^{\eps}$ that $\theta$ hits zero. Generally, define $T_{k+1}^{\eps}$ to be the first time after $S_k^{\eps}$ that $\theta$ reaches $\eps$ and define $S_{k+1}^{\eps}$ to be the first time after $T_{k+1}^{\eps}$ that $\theta$ hits zero. For $t>0$, suppose $g_t$ is the conformal map corresponding to the Loewner chain in the definition of $\eta$ and 
denote by $x_t$ the preimage of $V_t$ under $g_t$. Then, by the above target-independence, we see that, for each $k\ge 1$, the conditional law of $(\eta(t), T_k^{\eps}\le t\le S_k^{\eps})$ given $(\eta(t), t\le T_k^{\eps})$ is the same as $\SLE_{\kappa}$ in $\HH\setminus\eta[0,T_k^{\eps}]$ from $\eta(T_k^{\eps})$ to $x_{T_k^{\eps}}$ up to the disconnection time. In other words, the  conditional law of $(g_{T_k^{\eps}}(\eta(t)), T_k^{\eps}\le t\le S_k^{\eps})$ is the same as $\SLE_{\kappa}$ in $\HH$ from $0$ to $\eps$ up to the disconnection time. Roughly speaking, when $\kappa\in (4,8)$, $\SLE_{\kappa}(\kappa-6)$ can be constructed by concatenating a sequence of i.i.d. $\SLE_{\kappa}$ excursions (see Sections~\ref{subsec::pre_approximate_Bessel} and~\ref{subsec::pre_chordalchain}). In particular, we have the following decomposition of $\theta$: $\{(\theta(t)/\sqrt{\kappa}, T_k^{\eps}\le t\le S_k^{\eps})\}_k$ are i.i.d. Bessel excursions. Each of them is a Bessel process starting from $\eps$ and stopped at the first hitting time of zero.  

In our approach of proving Theorem~\ref{thm::cvg_interface_chordal}, we wish to follow the above excursion construction of $\SLE_{\kappa}(\kappa-6)$. Suppose the same setup as in Theorem~\ref{thm::cvg_interface_chordal} and suppose $\gamma^n$ is the interface in $\Omega^{\delta_n}$ from $a^{\delta_n}$ to $c^{\delta_n}$ where the lattice size $\delta_n\to 0$. Fix some conformal map $\phi^n: (\Omega^{\delta_n}; a^{\delta_n}, c^{\delta_n})\to (\HH; 0, \infty)$ (resp. $\phi: (\Omega; a, c)\to (\HH; 0, \infty)$) such that $\phi^n(b^{\delta_n})\to \phi(b)$. Denote by $\eta^n=\phi^n(\gamma^n)$. Denote by $W^n$ its driving function and $\theta^n$ the renormalized harmonic measure of the right side of $\eta^n[0,t]$ union $[0,\phi^n(b^{\delta_n})]$ seen from infinity.
The goal of Theorem~\ref{thm::cvg_interface_chordal} is to show the convergence of $\eta^n$ to $\eta\sim\SLE_{\kappa}(\kappa-6)$ in distribution. To this end, we first introduce stopping times $T_k^{n,\eps}$ and $S_k^{n,\eps}$ for $\eta^n$ which are the analogs of $T_k^{\eps}$ and $S_k^{\eps}$ for $\eta$, see Section~\ref{subsec::key_setup}. These stopping times decompose $\theta^n$ into excursions $\{(\theta^n(t), T_k^{n,\eps}\le t\le S_k^{n,\eps})\}_k$ and dusts $\{\theta^n(t), S_k^{n,\eps}\le t\le T_{k+1}^{n,\eps}\}_k$. Our strategy is as follows.
\begin{enumerate}
\item First, we argue that $\{\eta^n\}_n$ is tight, see details in Section~\ref{subsec::cvg_curves_chordal}. For any convergent subsequence, which we still denote by $\{\eta^n\}_n$, we know that the limiting process $\eta$ is a continuous curve with continuous driving function $W$. Moreover, $W^n\to W$ and $\eta^n\to \eta$ locally uniformly. The key ingredient in the first step is the topological framework developed in \cite{KemppainenSmirnovRandomCurves} and  Russo-Symour-Welsh bounds proved in \cite{ChelkakDuminilHonglerCrossingprobaFKIsing}. 
\item Second, we argue that $\theta^n\to \theta$ locally uniformly. This fact seems intuitive, but it is not as easy as one expects. We prove the convergence in Section~\ref{s.RHM}.
\item Third, we argue that the stopping times converge: $T_k^{n,\eps}\to T_k^{\eps}, S_k^{n,\eps}\to S_k^{\eps}$. Although we have $\eta^n\to \eta, W^n\to W$ and $\theta^n\to \theta$ locally uniformly, the convergence of the stopping times still requires certain technical works. One difficulty one faces is that one cannot rely on a stopping time for the limiting curve without possibly ruining the domain Markov property for the discrete exploration paths, see discussions in Remark~\ref{rem::ruiningMarkovproperty}. It will be proved in Section~\ref{ss.ST}. 
\item Fourthly, we use the convergence of the interface with Dobrushin boundary condition \cite{ChelkakSmirnovIsing} to conclude that, for each $k\ge 1$, the process $(\theta(t), T_k^{\eps}\le t\le S_k^{\eps})$ is a Bessel excursion and it is independent of $(\theta(t), t\le T_k^{\eps})$. There are several subtleties in this step. The first one is that, although $\theta^n\to \theta$, $T_k^{n,\eps}\to T_k^{\eps}$ and $S_k^{n,\eps}\to S_k^{\eps}$, we still need to control the processes on the intervals $[T_k^{n,\eps}\wedge T_k^{\eps}, T_k^{n,\eps}\vee T_k^{\eps}]$. The second one is that the Markov property of $\eta^n$ or $\theta^n$ does not pass to the limit $\eta$ or $\theta$ automatically. This is related to the convergence of the conditional distributions which can be quite delicate to conclude in general. See discussions in Section~\ref{subsec::cvg_oneBesselexcursion}.  
\item
Fifthly, we show that $\theta$ is a Bessel process. There are two natural ways to characterize the limiting process as a Bessel. They lead to the following two strategies A) and B):
\bi
\item[A)] Either we control the dusts $(\theta^n(t), S_k^{n,\eps}\le t\le T_{k+1}^{n,\eps})$ in a uniform way and argue that the dusts will disappear as $\eps\to 0$. In this strategy, the uniform control on the dusts requires certain technical works where the strong RSW \cite{ChelkakDuminilHonglerCrossingprobaFKIsing} plays an essential role.
\ei
This was our initial proof (see the unpublished manuscript \cite{Manuscript} for this approach) until we discovered \cite[Appendix A]{BenoistHonglerIsingCLE}. Their appendix made us realize that there is another way to characterize a Bessel process---Lemma~\ref{lem::approximate_bessel}---which leads to the following shorter strategy B) which we shall follow in this preprint.
\bi
\item[B)] Combining the previous step with the fact that $\Leb\{t: \theta(t)=0\}=0$, Lemma~\ref{lem::approximate_bessel} guarantees that $\theta$ is a Bessel process. 
\ei

\item Finally, we argue that $\eta$ is an $\SLE_{\kappa}(\kappa-6)$. In other words, we wish to argue that $W_t$ solves the SDE system~\eqref{eqn::slekappaminussix_sde} from the fact that $\theta$ is a Bessel process. Recall that $W$ is the driving function of $\eta$ and $\theta$ is defined as the renormalized harmonic measure of the right side of $\eta$. It is not immediate how to get information on $W$ out of $\theta$. Naively, the first trivial attempt is through the convergence from the discrete to the continuum, as one has in the discrete
\[\theta^n(t)=W_t^n+\int_0^t\frac{2ds}{\theta^n(s)}.\]
Combining with the facts that $W^n\to W$ and $\theta^n\to \theta$, it is tempting to conclude 
\begin{equation}
\label{eqn::theta_W}
\theta(t)=W_t+\int_0^t\frac{2ds}{\theta_s}. 
\end{equation}
However, it is not hard to find examples where the integral term does not necessarily converge as $W^n\to W$ and $\theta^n\to \theta$. In fact, the relation~\eqref{eqn::theta_W} still holds, but we prove it using the fact that $\eta$ is a continuous curve with continuous driving function $W$ and that $\eta$ satisfies the Russo-Symour-Welsh bounds. With~\eqref{eqn::theta_W} at hand, one can conclude that $\eta$ is indeed an $\SLE_{\kappa}(\kappa-6)$, see Section~\ref{subsec::cvg_besselprocess}.
\end{enumerate}

\medbreak
\textbf{\bf Acknowledgments:}
We wish to thank Vincent Beffara, Dimtry Chelkak, Antti Kemppainen and Avelio Sep\'{u}lveda, and Hugo Vanneuville for useful discussions.
We thank Jean-Christophe Mourrat for the proof of Lemma~\ref{lem::approximate_bessel}.  This work was carried out during visits of H.W. in Lyon funded by the ERC LiKo 676999. We thank an anonymous referee for helpful comments on the draft of this article. 

\section{Preliminaries}

\subsection{Approximate Bessel process}\label{subsec::bessel}
\label{subsec::pre_approximate_Bessel}
Let $(X_t, t\ge 0)$ be a Bessel process of dimension $d>0$. 
See \cite[Chapter XI]{RevuzYorMartBM} for the definition and properties. We focus in $d>1$ in this article. 
When $d>1$, it is a semimartingale and a strong solution to the SDE:
\begin{equation}\label{eqn::sde_bessel}
X_t=X_0+B_t+\frac{d-1}{2}\int_0^t\frac{ds}{X_s},
\end{equation}
where $B$ is a standard one-dimensional Brownian motion. When $d\in (1,2)$, it almost surely assumes the value zero on a nonempty random set with zero Lebesgue measure. Standard excursion theory (see e.g. \cite[Chapter IV]{BertoinLevyProcesses}) shows that if we decompose $X$ according to zero points, then it gives a Poisson point process of Bessel excursions of the same dimension.   

Fix $d\in (1,2)$, let $X=(X_t, t\ge 0)$ be a Bessel process of dimension $d$ starting from zero. 
We will decompose the process according to zero points. For $\eps>0$, define sequences of stopping times: set $S_0^{\eps}=0$, for $k\ge 0$,
\[T_{k+1}^{\eps}=\inf\{t>S_k^{\eps}: X_t\ge \eps\},\quad S_{k+1}^{\eps}=\inf\{t>T_{k+1}^{\eps}: X_t=0\}.\]
We know that 
\[\left(X_t, S^{\eps}_{k}\le t\le T^{\eps}_{k+1}\right),\quad k\ge 0\]
are i.i.d; and that
\[\left(X_t, T^{\eps}_{k}\le t\le S^{\eps}_{k}\right),\quad k\ge 0\]
are i.i.d and their common law is Bessel excursion of dimension $d$ starting from $\eps$ and stopped when it hits zero. It turns out the behavior of $X$ on the intervals $\cup_k (T_k^{\eps}, S^{\eps}_k)$ is sufficient to characterize the whole process under mild assumptions.

Precisely, suppose $(\tilde{X}_t, t>0)$ is a continuous process with the following properties. 
Set $\tilde{S}_0^{\eps}=0$. 
Let $\tilde{T}^{\eps}_1$ be the first time that the process exceeds $\eps$. 
After $\tilde{T}^{\eps}_1$, the process evolves according to~\eqref{eqn::sde_bessel} until it hits zero at time $\tilde{S}^{\eps}_1$. 
For $k\ge 0$, let $\tilde{T}^{\eps}_{k+1}$ be the first time after $\tilde{S}^{\eps}_k$ that the process exceeds $\eps$. After $\tilde{T}^{\eps}_{k+1}$, the process evolves according to~\eqref{eqn::sde_bessel} until it hits zero at time $\tilde{S}^{\eps}_{k+1}$.

\begin{lemma} \label{lem::approximate_bessel}
Fix $d\in (1,2)$. Suppose $\tilde{X}$ satisfies the following assumptions. 
\begin{enumerate}
\item[(1)] For $\eps>0$, for each $k\ge 1$, 
the processes 
$(\tilde{X}_t, t\le \tilde{T}_{k}^{\eps})$ and $(\tilde{X}_t, \tilde{T}_{k}^{\eps}\le t\le \tilde{S}_{k}^{\eps})$ 
are independent. 
\item[(2)] $\Leb\{t: \tilde{X}_t=0\}=0$ almost surely.
\end{enumerate}
Then the process $\tilde{X}$ is a Bessel process of dimension $d$. 
\end{lemma}

In the literature, this lemma seems to be well known. However, we could not find a proper reference for its proof. We include the following proof due to Jean-Christophe Mourrat for completeness.
\begin{proof}
Suppose $(X_t, t\ge 0)$ is a Bessel process of dimension $d$. We couple the two processes $X$ and $\tilde{Y}^{\eps}$ in the following way: for each $\eps>0$, for each $k\ge 0$, we set 
\begin{align*}
\tilde{Y}^{\eps}_{t+\tilde{S}_k^{\eps}}
=\tilde{X}_{t+\tilde{S}_k^{\eps}},
\quad & 0\le t\le \tilde{T}^{\eps}_{k+1}-\tilde{S}_k^{\eps};\\ 
\tilde{Y}^{\eps}_{t+\tilde{T}_{k+1}^{\eps}}=X_{t+T_{k+1}^{\eps}},
\quad & 0\le t\le S_{k+1}^{\eps}-T_{k+1}^{\eps}.
\end{align*}
In fact, $\tilde{Y}^{\eps}$ has the same law as $\tilde{X}$, and it coincides with $X$ on $[0,t]$ up to translation of time by an amount of at most $A_{\eps}(t)+\tilde{A}_{\eps}(t)$ where 
\[A_{\eps}(t):=\{s\le t: X_s\le \eps\},\quad \tilde{A}_{\eps}(t):=\{s\le t: \tilde{X}_s\le \eps\}.\] 
Suppose $\tilde{Y}$ is any sub-sequential limit of $\tilde{Y}^{\eps}$. Then $\tilde{Y}$ has the same law as $\tilde{X}$, and $\tilde{Y}\equiv X$ because, almost surely, 
\[\lim_{\eps\to 0}\Leb(A_{\eps}(t))=\Leb\{s\le t: X_s=0\}=0,\quad \lim_{\eps\to 0}\Leb(\tilde{A}_{\eps}(t))=\Leb\{s\le t: \tilde{X}_s=0\}=0.\]
These give that $\tilde{X}$ has the same law as a Bessel process of dimension $d$.
\end{proof}

\subsection{Chordal Loewner chain}
\label{subsec::pre_chordalchain}

Suppose that $K$ is a is compact subset of $\overline{\HH}$. We call $K$ an \textbf{$\HH$-hull} if $K=\overline{\HH\cap K}$ and $H=\HH\setminus K$ is simply connected. By Riemann's mapping theorem, there exists a unique  conformal map $\Psi$ from $H$ onto $\HH$ such that
$\Psi(z)=z+O(1/z)$, as $z\to\infty$.
In particular, there exists $a=a(K)\ge 0$ such that
\[\Psi(z)=z+2a/z+o(1/z),\quad\text{as }z\to\infty.\]
The quantity $a(K)$ is a non-negative increasing function of the set $K$, and we call it the \textit{half-plane capacity} of $K$ and denote it by $\hcap(K)$. 

We list some some estimates of the half-plane capacity which are useful in the later sections. For their proof, see for instance\cite[Lemma A.13]{KemppainenSmirnovRandomCurves}. 

\begin{lemma}\label{l.Hcap} $ $
\begin{itemize}
\item If $K\subset B(x,\eps)$ for some $x\in\R$, then $\hcap(K)\le \eps^2/2$. 
\item If $K\cap (\R\times\{\eps i\})\neq\emptyset$, then $\hcap(K)\ge \eps^2/4$.
\item If $K\subset [-l,l]\times[0,\eps]$, then $\hcap(K)\le \frac{l\eps}{2\pi}(1+o(1))$ as $\eps/l\to 0$.
\end{itemize}

\end{lemma}

Given a continuous function $W: [0,\infty)\to \R$, consider the solution for the following ODE: for $z\in \overline{\HH}$,
\[\partial_t g_t(z)=\frac{2}{g_t(z)-W_t},\quad g_0(z)=z.\]
This solution is well-defined up to the swallowing time 
\[T(z):=\inf\big\{t: \inf_{s\in [0,t]}|g_s(z)-W_s|>0\big\}.\]
For $t\ge 0$, define $K_t:=\{z\in\overline{\HH}: T(z)\le t\}$, then $g_t(\cdot)$ is the unique conformal map from $\HH\setminus K_t$ onto $\HH$ with the expansion $g_t(z)=z+2t/z+o(1/z)$ as $z\to\infty$. We call $(K_t, t\ge 0)$ the \textbf{chordal Loewner chain} corresponding to the driving function $(W_t, t\ge 0)$. 

We record three lemmas~\ref{lem::diamhull_deterministic} to~\ref{lem::chordal_evolution_boundary_point} about deterministic properties as follows, which will be useful later in the paper. 
\begin{lemma}\label{lem::diamhull_deterministic}
\cite[Lemma~2.1]{LawlerSchrammWernerLERWUST}.
There is a constant $C>0$ such that the following holds. Let $(W_t, t\ge 0)$ be a continuous driving function and $(K_t, t\ge 0)$ be the corresponding Loewner chain. Set 
\[k(t):=\sqrt{t}+\max\{|W(s)-W(0)|: s\in [0,t]\}.\]
Then, for all $t\ge 0$, we have 
\[C^{-1}k(t)\le\diam(K_t)\le Ck(t). \]
\end{lemma}

\begin{lemma}\label{lem::chordal_evolution_zeroLeb}
\cite[Lemma 2.5]{MillerSheffieldIG1}.
Suppose that $\eta$ is a continuous path in $\overline{\HH}$ from 0 to $\infty$ that admits a continuous Loewner driving function $W$. Then $\Leb\{t: \eta(t)\in\R\}=0$. 
\end{lemma}

\begin{lemma}\label{lem::chordal_evolution_boundary_point}
\cite[Lemma 3.3]{MillerSheffieldIG2}.
Suppose that $\eta$ is a continuous path in $\overline{\HH}$ from 0 to $\infty$ that admits a continuous Loewner driving function $W$. Let $(g_t)$ be the corresponding family of conformal maps. For each $t$, let $X_t$ be the right most point of $g_t(\eta[0,t])\cap\R$. If $\Leb(\eta\cap \R)=0$, then $X$ solves the integral equation
\begin{equation}\label{eqn::evolution_forcepiont_determinisitc}
X_t=\int_0^t\frac{2ds}{X_s-W_s},\quad X_0=0^+.
\end{equation}
\end{lemma}

\textbf{Chordal $\SLE_{\kappa}$} is the chordal Loewner chain $K$ with driving function $W=\sqrt{\kappa}B$ where $B$ is a one-dimensional Brownian motion. For $\kappa>0$, $\SLE_{\kappa}$ is almost surely a continuous transient curve: there exists a continuous curve $\eta$ in $\HH$ from $0$ to $\infty$ such that $\HH\setminus K_t$ is the unbounded connected component of $\HH\setminus\eta[0,t]$ for all $t$. When $\kappa\in [0,4]$, the curve is simple; when $\kappa\in (4,8)$, the curve is self-touching; and when $\kappa\ge 8$, the curve is space-filling. (See \cite{LawlerConformallyInvariantProcesses} and references therein). 

\textbf{Chordal $\SLE_{\kappa}(\rho)$} with force point $V_0=x\in\R$ is the chordal Loewner chain with driving function $W$ solving the following SDEs:
\begin{equation}\label{eqn::chordal_sle_sde}
dW_t=\sqrt{\kappa}dB_t+\frac{\rho dt}{W_t-V_t},\quad W_0=0; \quad dV_t=\frac{2dt}{V_t-W_t},\quad V_0=x. 
\end{equation}
For $\kappa>0$ and $\rho>-2$, define $\theta_t=V_t-W_t$. The process $\theta_t/\sqrt{\kappa}$ is a Bessel process of dimension $1+2(\rho+2)/\kappa>1$, hence $V_t-W_t$ is well-defined for all times. This implies the existence and uniqueness of the solution to~\eqref{eqn::chordal_sle_sde}. It is proved in \cite{MillerSheffieldIG1} that $\SLE_{\kappa}(\rho)$ with $\rho>-2$ is almost surely generated by continuous and transient curves. Suppose $(\Omega; a, b, c)$ is a simply connected domain $\Omega$ with three marked points (degenerate prime ends) $a, b, c$ on the boundary in counterclockwise order. We define $\SLE_{\kappa}(\rho)$ in $\Omega$ from $a$ to $c$ with force point $b$ as the image of $\SLE_{\kappa}(\rho)$ in $\HH$ from $0$ to $\infty$ with force point $1$ under the conformal map $\phi: (\HH; 0, 1, \infty)\to (\Omega; a, b, c)$. 
In this article, we are interested in $\SLE_{\kappa}(\kappa-6)$ as it has the following target-independent property. 

\begin{lemma}\label{lem::sle_targetindep}
\cite{SchrammWilsonSLECoordinatechanges}.
Suppose $\eta$ is an $\SLE_{\kappa}(\kappa-6)$ in $\Omega$ from $a$ to $c$ with force point $b$ and define $S$ to be the first time that $\eta$ hits the boundary arc $\partial_{bc}$, then $(\eta(t), 0\le t\le S)$ has the same law as an $\SLE_{\kappa}$ in $\Omega$ from $a$ to $b$ up to the first time that it hits the boundary arc $\partial_{bc}$.
\end{lemma}

Suppose $\eta$ is an $\SLE_{\kappa}(\kappa-6)$ in $\HH$ from $0$ to $\infty$ with force point $x\ge 0$. Then the process $\theta_t=V_t-W_t$ is the renormalized harmonic measure (see Definition~\ref{d.RHM}) of the right side of $\eta[0,t]$ union $[0,x]$ seen from infinity. On the other hand, we find 
\[d\theta_t=-\sqrt{\kappa}dB_t+\frac{(\kappa-4)dt}{\theta_t}.\]
Thus the process $\theta_t/\sqrt{\kappa}$ is a Bessel process of dimension $3-8/\kappa$. Note that, 
\[3-8/\kappa\in (1,2),\quad\text{when}\quad \kappa\in (4,8).\]

\subsection{Convergence of curves: the chordal case}
\label{subsec::cvg_curves_chordal}
In this section, we recall the main result of \cite{KemppainenSmirnovRandomCurves}. Let $X$ be the set of continuous oriented unparameterized curves, that is, continuous mappings from $[0,1]$ to $\C$ modulo reparameterization. We equip $X$ with the metric 
\begin{equation}\label{eqn::metric_space_curves}
d_X(\gamma_1, \gamma_2)=\inf_{\varphi_1, \varphi_2}\sup_{t\in [0,1]}|\gamma_1(\varphi_1(t))-\gamma_2(\varphi_2(t))|,
\end{equation}
where the infimum is over all increasing homeomorphisms $\varphi_1, \varphi_2: [0,1]\to [0,1]$. The topology on $(X, d_X)$ gives rise to a notion of weak convergence for random curves on $X$. 

We call $(\Omega; a, b)$ a Dobrushin domain if $\Omega$ is a bounded simply connected domain $\Omega$ with two distinct degenerate prime ends $a, b$ on the boundary.
We denote by $\partial_{ab}$ or $(ab)$ the boundary arc of $\partial\Omega$ from $a$ to $b$ in counterclockwise order. For instance, the unit disc $\U$ with two boundary points $(\U; -1, +1)$ is a Dobrushin domain. 

Let $X_{\simple}(\Omega; a, b)$ be the collection of continuous simple curves in $\Omega$ from $a$ to $b$ such that they only touch the boundary $\partial\Omega$ in $\{a, b\}$. In other words,  $X_{\simple}(\Omega; a, b)$ is the collection of continuous simple curves $\gamma$ such that
\[\gamma(0)=a,\quad \gamma(1)=b, \quad \gamma(0,1)\subset\Omega.\]  

Let $X_0(\Omega; a, b)$ be the closure of the space $X_{\simple}(\Omega; a, b)$ in the metric topology $(X, d_X)$. We often consider some reference sets $X_0(\U; -1, +1)$ and $X_0(\HH; 0, \infty)$ where the latter can be understood by extending the above definition to curves defined on the Riemann sphere.

Since choral $\SLE$ is invariant under scaling, we can define chordal $\SLE$ in $(\Omega; a, b)$ via conformal image: suppose $\phi$ is any conformal map from $\HH$ onto $\Omega$ that sends $0,\infty$ to $a, b$, we define chordal $\SLE$ in $\Omega$ from $a$ to $b$ by the image of chordal $\SLE$ in $\HH$ from $0$ to $\infty$ by $\phi$. Note that $\SLE_{\kappa}$ is in $X_{\simple}(\Omega; a, b)$ almost surely when $\kappa\le 4$  and it is in $X_0(\Omega; a, b)$ almost surely when $\kappa>4$. 

We call $(Q; x_1, x_2, x_3, x_4)$ a \textbf{quad} if $Q$ is simply connected subset of $\C$ with four distinct boundary points $x_1, x_2, x_3, x_4$. The four points are in counterclockwise order. We denote by $d_{Q}((x_1x_2), (x_3x_4))$ the extremal distance between $(x_1x_2)$ and $(x_3x_4)$ in $Q$. 
We say that a curve $\gamma$ crosses $Q$ if there exists a subinterval $[s,t]$ such that $\gamma(s,t)\subset Q$ and $\gamma[s,t]$ intersects both $(x_1x_2)$ and $(x_3x_4)$. 

For any curve $\gamma\in X_0(\Omega; a, b)$ and any time $\tau$,  define $\Omega(\tau)$ to be the connected component of $\Omega\setminus\gamma[0,\tau]$ with $b$ on the boundary. Consider a quad $(Q; x_1, x_2, x_3, x_4)$ in $\Omega(\tau)$ such that $(x_2x_3)$ and $(x_4x_1)$ are contained in $\partial \Omega(\tau)$. We say that $Q$ is \textbf{avoidable} if it does not disconnect $\gamma(\tau)$ from $b$ in $\Omega(\tau)$.    

\begin{definition}\label{def::c2_chordal}
Suppose $\{(\Omega_n; a_n, b_n)\}_n$ is a sequence of Dobrushin domains. For each $n$, suppose $\PP_n$ is a probability measure supported on $X_0(\Omega_n; a_n, b_n)$. We say that the collection $\Sigma(M)=\{\PP_n\}_n$ satisfies \textbf{Condition C2} if there exists a constant $M>0$ such that for any $\PP_n\in\Sigma(M)$, any stopping time $0\le \tau\le 1$, and any avoidable quad $(Q; x_1, x_2, x_3, x_4)$ of $\Omega_n(\tau)$ such that $d_{Q}((x_1x_2), (x_3x_4))\ge M$, we have
\[\PP_n[\gamma[\tau, 1] \text{crosses }Q\cond \gamma[0,\tau]]\le 1/2.\] 
\end{definition}

For a probability measure $\PP$ on curves in $\Omega$, let $\phi$ be a conformal map defined on $\Omega$. We denote by $\phi\PP$ the pushforward of $\PP$ by $\phi$. 
For the Dobrushin domain $(\Omega_n; a_n, b_n)$, let $\psi_n$ be any conformal map from $(\Omega_n; a_n, b_n)$ onto $(\U; -1, +1)$. 
Given the family $\Sigma(M)$ as above, define the family
\[ \Sigma_{\U}(M)=\{\psi_n\PP_n: \PP_n\in\Sigma(M)\}.\]

\begin{theorem}\label{thm::chordal_loewner_cvg}
If the family $\Sigma(M)$ satisfies Condition C2, then the family $\Sigma_{\U}(M)$ is tight in the topology induced by~\eqref{eqn::metric_space_curves}. 
Suppose $\PP_{\infty}$ is a limiting measure of the family $\Sigma_{\U}(M)$, then the following statements hold $\PP_{\infty}$ almost surely.
\begin{enumerate}
\item[(1)] There exists $\beta>0$ such that $\gamma$ has a H\"{o}lder-continuous parametrization for the H\"{o}lder exponent $\beta$. 
\item[(2)] The tip $\gamma(t)$  of the curve lies on the boundary of the connected component of $\U\setminus\gamma[0,t]$ with $1$ on the boundary for all $t$.
\item[(3)] The curve $\gamma$ is transient: i.e. $\lim_{t\to 1}\gamma(t)=1$. 
\end{enumerate}
Suppose $\gamma_n\sim\PP_n$ and let $\phi_n$ be any conformal map from $(\Omega_n; a_n, b_n)$ onto $(\HH; 0, \infty)$. We parameterize 
$\eta_n:=\phi_n(\gamma_n)$ by the half-plane capacity.
Let $W_n$ be the driving process of $\eta_n$. Then 
\begin{enumerate}
\item[(4)] $\{W_n\}_n$ is tight in the metrizable space of continuous function on $[0,\infty)$ with the topology of local uniform convergence. 
\item[(5)] $\{\eta_n\}_n$ is tight in the metrizable space of continuous function on $[0,\infty)$ with the topology of local uniform convergence.
\end{enumerate}  
Moreover, if the sequence converges in any of the topologies (4) and (5) above it also converges in the other topology
and the limits agree in the sense that the limiting curve is driven by the limiting driving process. 
\end{theorem}
\begin{proof}
\cite[Theorem~1.5, Corollary~1.7, Proposition~2.6]{KemppainenSmirnovRandomCurves}. 
\end{proof}

\begin{lemma}\label{lem::processvsrhm}
Assume the same as in Theorem~\ref{thm::chordal_loewner_cvg} and suppose $\{\eta_n\}_n$ is any convergent subsequence and $\eta$ is the limiting process. Then $\eta$ satisfies all the requirements in Lemma~\ref{lem::chordal_evolution_boundary_point} almost surely. 
\end{lemma}
\begin{proof}
Theorem~\ref{thm::chordal_loewner_cvg} guarantees that $\eta$ is generated by a continuous curve with a continuous Loewner driving function $W$. We only need to check that $\eta\cap\R$ has zero Lebesgue measure. For convenience, we couple all $\eta_n$ and $\eta$ in the same space so that $\eta_n\to\eta$ and $W_n\to W$ locally uniformly almost surely. 

It is sufficient to prove there exists $\alpha>0$ such that, for all $x\in\R$ and $\eps>0$, we have 
\begin{equation}\label{eqn::C2_onepointestimate}
\PP[\eta\cap B(x,\eps)\neq\emptyset]\le 10(\eps/|x|)^{\alpha}. 
\end{equation}
With~\eqref{eqn::C2_onepointestimate} at hand, we see that $\PP[\eta\text{ hits }x]=0$ for all $x\in\R\setminus\{0\}$, thus $\eta\cap\R$ has zero Lebesgue measure. 
We only need to prove~\eqref{eqn::C2_onepointestimate} for $x\ge 4\eps$. Let $A_x(r, R)$ be the semi-annulus in $\HH$ with center at $x$ and inradius $r$ and outradius $R$. It is proved in \cite[Proposition~2.6]{KemppainenSmirnovRandomCurves} that Condition C2 implies the following property: there exists $\alpha'>0$ such that, for any avoidable quad $Q$ in $\HH$, 
\[\PP[\eta_n\text{ crosses }Q]\le 10\exp(-\alpha' d_Q((x_1x_2), (x_2x_4))).\]
We apply this property to $Q=A_x(2\eps, x/2)$, then there exists $\alpha>0$ such that 
\[\PP[\eta_n\text{ crosses }A_x(2\eps, x/2)]\le 10(\eps/x)^{\alpha}.\]
For $T>0$, denote by 
$\|\eta_n-\eta\|_{\infty, T}=\inf\{|\eta_n(t)-\eta(t)|: 0\le t\le T\}$.
Then we have
\begin{align*}
\PP[\eta[0,T]\cap B(x,\eps)\neq\emptyset]&\le \PP[\eta_n\text{ crosses }A_x(2\eps, x/2)]+\PP[\|\eta_n-\eta\|_{\infty, T}\ge \eps]\\
&\le 10(\eps/x)^{\alpha}+\PP[\|\eta_n-\eta\|_{\infty, T}\ge \eps].
\end{align*}
Let $n\to\infty$, we have 
\[\PP[\eta[0,T]\cap B(x,\eps)\neq\emptyset]\le 10(\eps/x)^{\alpha}. \]
Thus 
\begin{align*}
\PP[\eta\cap B(x,\eps)\neq\emptyset]&\le \PP[\eta[0,T]\cap B(x,\eps)\neq\emptyset]+\PP[\eta\cap B(x,\eps)\neq\emptyset, \eta[0,T]\cap B(x,\eps)=\emptyset]\\
&\le 10(\eps/x)^{\alpha}+\PP[\eta\cap B(x,\eps)\neq\emptyset, \eta[0,T]\cap B(x,\eps)=\emptyset]. 
\end{align*}
Let $T\to\infty$, as $\eta$ is transient, the second term goes to zero and we obtain~\eqref{eqn::C2_onepointestimate}. This completes the proof. 
\end{proof}

\subsection{Application to exploration paths of FK-Ising percolation}
\label{subsec::pre_fk}
Let us start by recalling some useful facts on FK-Ising percolation that we will need later in this paper. 
The reader may consult \cite{DCParafermionic} for general background on FK-Ising percolation.

We will consider finite subgraphs $\Omega=(V(\Omega), E(\Omega))\subset \Z^2$. For such a graph, we denote by $\partial \Omega$ the inner boundary of $\Omega$:
\[\partial \Omega=\{x\in V(\Omega): \exists y\not\in V(\Omega) \text{ such that } \{x,y\}\in E(\Z^2)\}.\]  
A configuration $\omega=(\omega_e: e\in E(\Omega))$ is an element of $\{0,1\}^{E(\Omega)}$. If $\omega_e=1$, the edge $e$ is said to be open, otherwise $e$ is said to be closed. The configuration $\omega$ can be seen as a subgraph of $\Omega$ with the same set of vertices $V(\Omega)$, and the set of edges given by open edges $\{e\in E(\Omega): \omega_e=1\}$.  

Given a finite subgraph $\Omega\subset \Z^2$, a boundary condition $\xi$ is a partition $P_1\sqcup \cdots\sqcup P_k$ of $\partial \Omega$. Two vertices are wired in $\xi$ if and only if they belong to the same $P_i$. The graph obtained from the configuration $\omega$ by identifying the wired vertices together in $\xi$ is denoted by $\omega^{\xi}$. Boundary conditions should be understood informally as encoding how sites are connected outside of $\Omega$. Let $o(\omega)$ and $c(\omega)$ denote the number of open can dual edges of $\omega$ and $k(\omega^{\xi})$ denote the number of maximal connected components of the graph $\omega^{\xi}$.  

The probability measure $\phi^{\xi}_{p,q,\Omega}$ of the \textit{random cluster model} model on $\Omega$ with edge-weight $p\in [0,1]$, cluster-weight $q>0$ and boundary condition $\xi$ is defined by 
\[\phi^{\xi}_{p,q,\Omega}[\omega]:=\frac{p^{o(\omega)}(1-p)^{c(\omega)}q^{k(\omega^{\xi})}}{Z^{\xi}_{p,q,\Omega}},\]
where $Z^{\xi}_{p,q,\Omega}$ is the normalizing constant to make $\phi^{\xi}_{p,q,\Omega}$ a probability measure. For $q=1$, this model is simply Bernoulli bond percolation. 

If all the vertices in $\partial \Omega$ are pairwise wired (the partition is equal to $\partial \Omega$), it is called \textit{wired boundary condition}. 
The random cluster model with wired boundary condition on $\Omega$ is denoted by $\phi^{1}_{p,q,\Omega}$. If there is no wiring between vertices in $\partial \Omega$ (the partition is composed of singletons only), it is said to be with dual-wired boundary condition or \textit{free boundary condition}. The random cluster model with free boundary condition on $\Omega$ is denoted by $\phi^{0}_{p,q,\Omega}$. 

We call critical \textit{FK-Ising} model the random cluster model with 
\[q=2, \quad p=p_c(2).\]
For $q=2$, we have a stronger version of the so-called RSW crossing estimates. Given a discrete quad $(Q; a, b, c, d)$, 
we denote by $d_{Q}((ab), (cd))$ the \textit{discrete extremal distance} between $(ab)$ and $(cd)$ in $Q$, see \cite[Section 6]{ChelkakRobustComplexAnalysis}. The discrete extremal distance is uniformly comparable to and converges to its continuous counterpart---the classical extremal distance (see e.g. \cite[Chapter~4]{AhlforsConformalInvariants}). 
\begin{theorem}\label{thm::rsw_strong}
\cite[Theorem 1.1]{ChelkakDuminilHonglerCrossingprobaFKIsing}. 
Fix $q=2$. For each $L>0$ there exists $c(L)>0$ such that, for any quad $(Q; a, b, c, d)$ and any boundary condition $\xi$, if $d_{Q}((ab), (cd))\le L$, then 
\[\phi^{\xi}_{p_c(2), 2, Q}\left[(ab)\leftrightarrow(cd)\right]\ge c(L).\]
\end{theorem}

\begin{figure}[ht!]
\begin{center}
\includegraphics[width=0.6\textwidth]{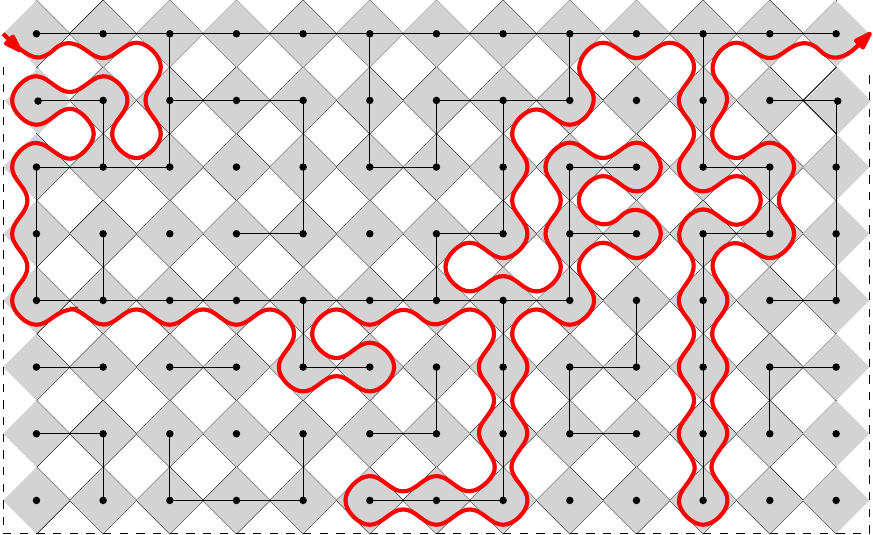}
\end{center}
\caption{\label{fig::fkinterface_Dobrushin}}
\end{figure}

The \textit{medial lattice} $(\Z^2)^{\diamond}$ is the graph with the centers of edges of $\Z^2$ as vertex set, and edges connecting nearest vertices. This lattice is a rotated and rescaled version of $\Z^2$. The vertices and edges of $(\Z^2)^{\diamond}$ are called medial-vertices and medial-edges. We identify the faces of $(\Z^2)^{\diamond}$ with the vertices of $\Z^2$ and $(\Z^2)^*$. A face of $(\Z^2)^{\diamond}$ is said to be black if it corresponds to a vertex of $\Z^2$ and white if it corresponds to a vertex of $(\Z^2)^*$. See more detail and figures in \cite[Section~3]{DCParafermionic}.

 Fix a Dobrushin domain $(\Omega; a, b)$ and consider a configuration $\omega$ together with its dual-configuration $\omega^*$. The \textit{Dobrushin boundary condition} is given by taking edges of $\partial_{ba}$ to be open and the dual-edges of $\partial_{ab}^*$ to be dual-open; in this case, we also say that the boundary condition along $\partial_{ba}$ is wired and the boundary condition along $\partial_{ab}$ is free. 
  Through every vertex of $\Omega^{\diamond}$, there passes either an open edge of $\Omega$ or a dual open edge of $\Omega^*$. Draw self-avoiding loops on $\Omega^{\diamond}$ as follows: a loop arriving at a vertex of the medial lattice always makes a $\pm\pi/2$ turn so as not to cross the open or dual open edges through this vertex. The loop representation contains loops together with a self-avoiding path going from $a^{\diamond}$ to $b^{\diamond}$, see Fig.~\ref{fig::fkinterface_Dobrushin}. This curve is called the \textit{exploration path} in $\Omega^{\diamond}$ from $a^{\diamond}$ to $b^{\diamond}$. For $\delta>0$, we consider the rescaled square lattice $\delta\Z^2$. The definitions of dual and medial Dobrushin domains extend to this context. 

\begin{figure}[ht!]
\begin{center}
\includegraphics[width=0.6\textwidth]{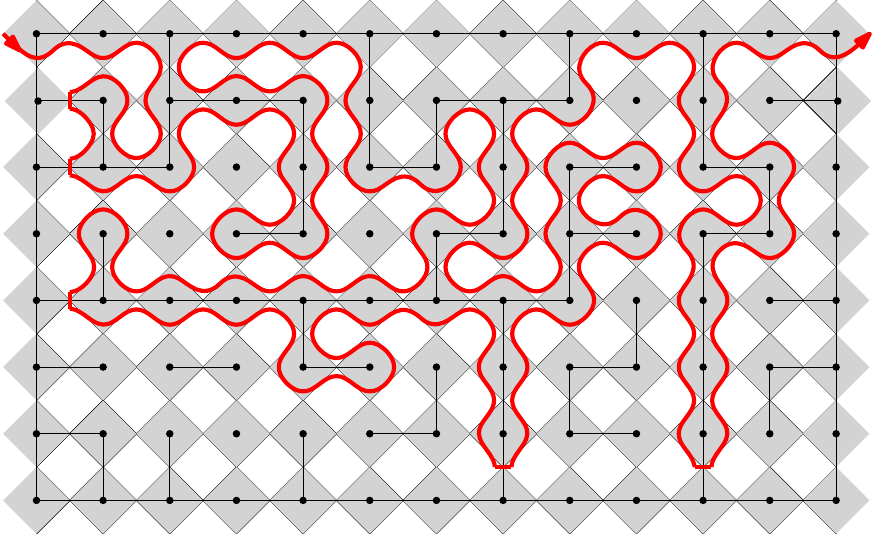}
\end{center}
\caption{\label{fig::fkinterface_wired}}
\end{figure}

\begin{theorem}\label{thm::fkising_cvg_Dobrushin}
\cite[Theorem~2]{CDCHKSConvergenceIsingSLE}.
Suppose $(\Omega; a, b)$ is a bounded simply connected subset $\Omega\subset\C$ with two distinct degenerate prime ends $a, b$ on the boundary. Let $(\Omega^{\delta}; a^{\delta}, b^{\delta})$ be a sequence of discrete Dobrushin domains on $\delta\Z^2$ converging to $(\Omega; a, b)$ in the Carath\'{e}odory sense: fix the conformal maps $\phi: (\Omega; a, b)\to (\HH; 0, \infty)$ and $\phi^{\delta}: (\Omega^{\delta}; a^{\delta}, b^{\delta})\to (\HH; 0, \infty)$ so that $\phi^{\delta}\to \phi$ as $\delta\to 0$ uniformly on compact subsets of $\Omega$. 
\begin{enumerate}
\item[(1)]Then the exploration path $\gamma^{\delta}$ of the critical FK-Ising model with Dobrushin boundary condition in $(\Omega^{\delta}; a^{\delta}, b^{\delta})$ converges in distribution for the topology induced by~\eqref{eqn::metric_space_curves} to chordal $\SLE_{16/3}$ in $\Omega$ from $a$ to $b$. 
\end{enumerate}
Suppose $\gamma$ is an $\SLE_{16/3}$ in $\Omega$ from $a$ to $b$. 
We parameterize $\phi^{\delta}(\gamma^{\delta})$ (resp. $\eta=\phi(\gamma)$) by the half-plane capacity and let $W^{\delta}$ (resp. $W$) be the driving function. Let $\delta_n\to 0$, denote by 
$\gamma^n:=\gamma^{\delta_n}$, $\eta^n:=\phi^{\delta_n}(\gamma^{\delta_n})$ and $W^{n}:=W^{\delta_n}$; and suppose $\{\eta^n\}$ is a convergent subsequence. We also have the followings 
\begin{enumerate}
\item[(2)] $W^n$ converges in distribution to $W$ with the topology of local uniform convergence.   
\item[(3)] $\eta^n$ converges in distribution to $\eta$ with the topology of local uniform convergence. 
\end{enumerate} 
\end{theorem}

In the above, we have defined the exploration path with Dobrushin boundary condition. Next, we will introduce the exploration path with wired boundary condition. 
Consider a configuration in $\Omega$ with wired boundary condition and draw its loop representation on $\Omega^{\diamond}$. 
Construct the exploration path from $a^{\diamond}$ to $b^{\diamond}$ as follows. Starting from $a^{\diamond}$, cut open the loop next to $a$ and follow the loop clockwise until one of the following two cases happens: (1) the path reaches the target; (2) the path arrives at a point which is disconnected from the target. If case (1) happens, the path stops. If case (2) happens, cut open the loop nearest to the current position and follow the new loop clockwise until one of the two cases happens, and repeat the same strategy. Continue in this way until the path reaches $b^{\diamond}$. Note that the exploration path continues so that it is not disconnected from the target and that primal cluster is on the left and dual cluster is on the right as long as it is possible. When it is not possible, the path continues so that dual cluster is on the right.  See Fig.~\ref{fig::fkinterface_wired}.

\subsection{Continuous quad vs. discrete quad}
\label{subsec::quad_continuous_discrete}

The strong RSW---Theorem~\ref{thm::rsw_strong}---will play an important role in the proof of Theorem~\ref{thm::cvg_interface_chordal}. Consider FK-Ising model in $\Omega^{\delta}$. Later in Section~\ref{sec::proof_main}, we will need to build appropriate quad in $\Omega^{\delta}$ to apply the strong RSW. However, the geometry of $\Omega^{\delta}$ can be arbitrary complicated, it is not straightforward to come up with quad inside $\Omega^{\delta}$. Our approach is as follows. 
\begin{enumerate}
\item We first map $\Omega^{\delta}$ onto $\HH$ using a conformal map $\phi^{\delta}$.
\item Then we build the appropriate quads in $\HH$.
\item Finally, and this is the main step, we argue that the image under $(\phi^{\delta})^{-1}$ of the continuous quad can be accurately approximated by discrete $\delta\Z^2$-quad in $\Omega^{\delta}$. This is the purpose of Lemma~\ref{l.tech} proved in this section. 
\end{enumerate}

Take $K=[0,2]\times[0,1]$ in $\HH$. Consider four scaled copies of such domains: $K_1=20K$, $\tilde{K}_1=10K, K_2=4K, \tilde{K}_2=2K$. See Fig.~\ref{f.keyNEW}.

\begin{figure}[!htp]
\begin{center}
\includegraphics[width=0.85\textwidth]{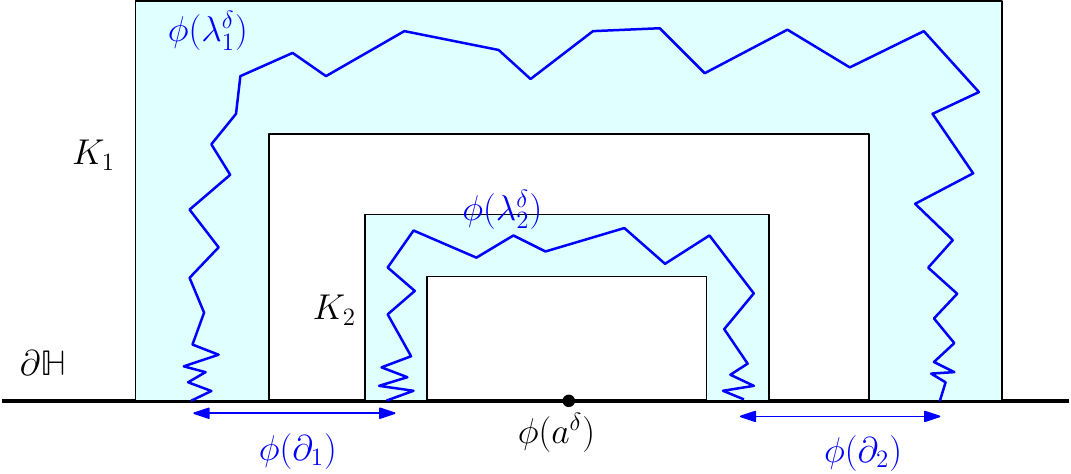}
\end{center}
\caption{}\label{f.keyNEW}
\end{figure}

\begin{lemma}\label{l.tech}
If $\delta$ is sufficiently small (i.e. $\delta \leq \delta_0(\eps,r,R)$), there is a lattice path $\lambda_1=\lambda_1^\delta$ in $\delta \Z^2 \cap \Omega^{\delta}$ which disconnects $\phi^{-1}(\partial K_1)$ from $\phi^{-1}(\partial \tilde K_1)$. Similarly there is a lattice path $\lambda_2=\lambda_2^\delta$ which disconnects  $\phi^{-1}(\partial K_2)$ from $\phi^{-1}(\partial \tilde K_2)$.
See Fig.~\ref{f.keyNEW}.
\end{lemma}

Note here that there are no issues of subtle prime ends as $\Omega^\delta$ is a $\delta\Z^2$ domain.

\begin{proof}
The proof relies on easy considerations of harmonic measure. Suppose one cannot find a path disconnecting say  $\phi^{-1}(\partial K_1)$ from $\phi^{-1}(\partial \tilde K_1)$. This means that one can necessarily find a square $Q_\delta$ in $\delta \Z^2$ of side-length $3 \delta$ which intersects $\phi^{-1}(\partial K_1)$ as well as $\phi^{-1}(\partial \tilde K_1)$. As such, the conformal image $\phi(Q_\delta)$ intersects $\partial K_1$ and $\partial \tilde K_1$ and its diameter needs to be larger than $\dist(\partial K_1, \partial \tilde K_1) \geq \eps$. In particular the harmonic measure of $\phi(Q_\delta)$ seen from 0 in $\U$ (for the Brownian motion stopped when first exiting $\U \setminus f(Q_\delta)$) is larger than $\frac 1 {100} \eps$ (by easy considerations on Brownian motion). Now, by conformal invariance of harmonic measure, the harmonic measure of the square $Q_\delta$ seen from 0 (for the B.M. stopped when first exiting $\Omega^\delta \setminus Q_\delta$) needs to be larger than $\frac 1 {100} \eps$ as well. On the other hand, as  $\Omega^\delta$ is bounded ($\Omega^\delta\subset B(0,R)$), by monotony properties of harmonic measure, the above harmonic measure is smaller than the harmonic measure in $B(0,R)$ of $Q_\delta$ seen from 0. As the distance from $Q_\delta$ to the origin is larger than $\Omega(r)$ (which follows for example from K\"obe's 1/4 theorem), this later harmonic measure is smaller than $ C_{R,r} (\log\frac{1}{\delta})^{-1}$.
\end{proof} 

\begin{remark}\label{r.suboptimal}
Note that it is possible to obtain much better bounds on how small $\delta$ needs to be (with slightly more technical proofs though, this is why we sticked to that one). For example one way is  to consider the extremal length in the annulus $A_1:= K_1 \setminus \tilde K_1$ from one of the arcs of $A_1$ intersecting $\partial \U$ to the other symmetric arc. This extremal length is clearly bounded from above by some constant $M<\infty$. If a path as in Lemma~\ref{l.tech} did not exist, then by designing an appropriate $\rho$-intensity on $\phi^{-1}(A_1)$ and using Beurling's estimate together with K\"{o}be's 1/4 Theorem, one can show that the extremal length (which is conformally invariant) would need to be larger than 
$\Omega(1) \log(\frac {\eps^2}{20 \delta})$ ($\eps^2$ comes from Beurling here)  which would yield a much better control on $\delta=\delta(\eps)$ in Lemma~\ref{l.tech}. 
\end{remark}

\subsection{On degenerate prime ends}

In the course of our proof, we will need to apply the above convergence Theorem \ref{thm::fkising_cvg_Dobrushin} at multiple occasions along the exploration procedure. In order to apply Theorem~\ref{thm::fkising_cvg_Dobrushin}, we need to make sure that the tip of the exploration path ($a^\delta \to a$), as well as the marked point at the end of the dual arc ($b^\delta \to b$) are degenerate prime ends a.s. This will follow from the following general Lemma.

\begin{lemma}\label{lem::degenerate_prime_end}
Let $\Omega$ be a bounded Jordan domain with some interior point $x_0$ and let $\gamma:[0,T] \to \bar{\Omega}$ be a continuous curve which avoids $x_0$. Then for any $t>0$, the conformal map $f: \U \to \Omega(t)$ (where $\Omega(t)$ is the connected component of $x_0$ in $\Omega\setminus \gamma([0,t])$ can be continuously extended to $\bar{\U}$. In particular all points on $\partial\Omega(t)$ are degenerate prime ends. (N.B. $\Omega(t)$ may not be a Jordan domain anymore). 
\end{lemma}

\begin{remark}\label{}
Note that this statement is similar in flavour to the {\em visibility of the tip} statement in Theorem~\ref{thm::chordal_loewner_cvg}--Item(2). But it is independent of it : it does not follow from, nor imply the {\em visibility of the tip} property. 
\end{remark}

This Lemma is not new: see for example Example 3.8 in \cite{LawlerConformallyInvariantProcesses} and its proof in \cite[pp 88--89]{NewmanElementsTopology}. We include a proof below for completeness. 
\begin{proof}[Proof of Lemma~\ref{lem::degenerate_prime_end}]
Following \cite[Proposition~3.7]{LawlerConformallyInvariantProcesses} (see also the continuity theorem p18 in \cite{Pommerenke}), it is equivalent to the fact that $\C\setminus \Omega(t)$ is locally connected for any time $t>0$. 
As we assumed that $\Omega$ is a Jordan domain, $\Omega(t)$ is the connected component of $x_0$ of a continuous curve $\eta: [0,1] \to \C$. Following \cite{LawlerConformallyInvariantProcesses}, a closed set $K\subset \C$ is  {\em locally connected} if for any $\eps>0$ there exists $\delta >0$ such that for any $z,w \in K$ with $|z-w|< \delta$, there exists a connected set $K_1\subset K$ with $z,w\in K_1$ and $\diam(K_1) \leq \eps$. 

Suppose $\C\setminus \Omega(t)$ is not locally connected, i.e. one can find $\eps>0$ and a sequence $\{z_n, w_n\}$ of points such that $|z_n-w_n|\to 0$ which do not satisfies the above property. As $\Omega$ is a bounded set, $\sup_n |z_n| \vee |w_n| <\infty$, we can then extract a convergent subsequence $\{\hat{z}_n, \hat{w}_n\}$ such that $\hat{z}_n, \hat{w}_n\to x^*$. If that point $x^*$ is at positive distance from the curve $\eta$ (in other words if $x^*$ is in $\C\setminus \eta([0,1])$), then it is immediate to reach a contradiction as $\C\setminus \eta([0,1])$ is open. If on the other hand, $x^*$ belongs to the range of $\eta$, then $x^* = \eta (t^*)$ for some $t^*$ and one can find times $u_n,$ (resp. $v_n$) such that $\hat z_n$ (resp. $\hat w_n$) is very close to $\eta (u_n)$ (resp. $\eta(v_n)$) in $\C \setminus \Omega(t)$ (for example by taking the closest points from $\hat z_n,\hat w_n$ to $\eta$).  By extracting further, we can assume $u_n \to u$ and $v_n \to v$. Our hypothesis implies that there is no connected subset $K_1$ in $\C \setminus \Omega(t)$ of diameter less than $\eps/2$ connecting $\eta(u)$ to $\eta(v)$. As $\hat{z}_n, \hat{w}_n \to x^*=\eta(t^*)$, by the continuity of the curve $\eta$, we must have $\eta(u) = \eta(v) = x^*$. This gives us a contradiction by choosing $K_1:= \{x^*\}$. 
\end{proof}

\section{Convergence of renormalized harmonic measure}\label{s.RHM}
Throughout this section, we will assume we are in the same setup as in Theorem~\ref{thm::fkising_cvg_Dobrushin}. We thus have a sequence of domains $(\Omega^n; a^n, b^n)$ ($=(\Omega^{\delta_n}; a^{\delta_n}, b^{\delta_n})$) which converge in Carath\'{e}odory sense to $(\Omega; a ,b)$ and we are given conformal maps  $\phi^n : \Omega^n \to \H$ and $\phi : \Omega \to \H$ satisfying the hypothesis in Theorem~\ref{thm::fkising_cvg_Dobrushin}. Recall the main convergence result from that Theorem is its item (3) on the random curves in $\H$, $\eta^n:=\phi^n(\gamma^n)$ and $\eta:=\phi(\gamma)$ each parametrised by the half-plane capacity in $\H$. Furthermore recall that the Loewner driving function $W$ of $\eta$ is the limit in law of $W^n$, the driving function of $\eta^n$. 

\begin{definition}\label{d.RHM}
For any Borel set $A\subset \R$, we define the \textbf{renormalized harmonic measure} of $A$ seen from infinity to be 
\begin{align*}\label{}
\mathrm{RHM}(A)=\mathrm{RHM}_{\H}(A):= \lim_{y \to \infty} \pi y  \; \FK{}{i\cdot y}{B_\tau \in A}\,,
\end{align*}
where $\tau$ is the first time the Brownian motion started at $i y$ touches $\p \H$. 
The multiplicative factor $\pi$ is there so that $\mathrm{RHM}([0,L])=L$.
By conformal invariance of Brownian motion, we define in the same fashion the \textbf{renormalized harmonic measure} for general hulls $H:= \H \setminus K$ where $K$ is any compact set of the plane as follows: for any subset $A\subset \p H$, we define
\begin{align*}\label{}
\mathrm{RHM}_H(A):= \lim_{y \to \infty} \pi y  \; \FK{}{i\cdot y}{B_{\tau^H} \in A}\,.
\end{align*}
\end{definition}

We now state the main result of this section.

\begin{proposition}\label{pr.thetan}
Assume we are in the same setup as in Theorem~\ref{thm::fkising_cvg_Dobrushin}.  
We also assume (using Skorokhod's representation theorem) that the random curves $\eta^n$ and $\eta$ (each parametrized by half-plane capacity) are coupled on the same probability space so that both $W^n$ and $\eta^n$ a.s. converge locally uniformly to  $W$ and $\eta$.
Let $t \mapsto \theta^n(t)$ (resp. $t\mapsto \theta(t)$) denote the renormalized harmonic measure of the right boundary of $\eta^n([0,t])$ (resp. $\eta([0,t])$) seen from infinity. 
Then $\theta^n$ a.s. converges to $\theta$ locally uniformly. I.e. for any $T>0$, almost surely
\begin{align*}\label{}
\| \theta^n  - \theta \|_{\infty, T} = \sup_{t\in [0,T]} |\theta^n(t) - \theta(t)| \to 0,\quad \text{as }n\to \infty.
\end{align*}
\end{proposition}

\begin{proof}
Let us fix some time $T>0$. Recall we are in the setup of Theorems~\ref{thm::chordal_loewner_cvg} and~\ref{thm::fkising_cvg_Dobrushin}. By combining hypothesis of Theorem~\ref{thm::chordal_loewner_cvg} with the estimate Lemma~\ref{lem::diamhull_deterministic}, one easily obtains that 
\begin{align}\label{e.diam}
M:= \sup_n \diam (\eta^n[0,T]) 
\vee 
\diam (\eta[0,T]) <\infty \quad\text { a.s.} 
\end{align}

\begin{figure}[!htp]
\begin{center}
\includegraphics[width=0.9\textwidth]{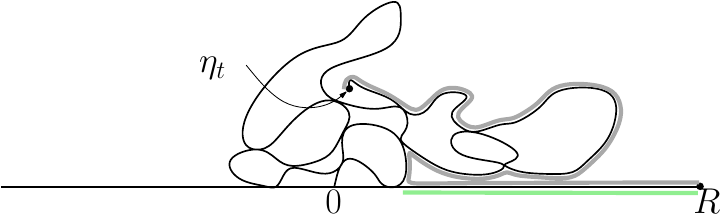}
\end{center}
\caption{We express the harmonic measure on the right of the curve $\eta([0,t])$ as the difference of the harmonic measure of the grey and green arcs.}\label{f.RHM1}
\end{figure}

Our proof will be based on writing the renormalized harmonic function $\theta^n(t)$ as a difference of two quantities: indeed one has for any $t\in [0,T]$, and for any $R > M$, 
\begin{align}\label{e.thetan}
\theta^n(t) = \mathrm{RHM}_{\H \setminus \eta^n[0,t]}(\eta^n(t), R) - \mathrm{RHM}_{\H \setminus \eta^n[0,t]}([0,R]) .
\end{align}
Here $\mathrm{RHM}_{\H \setminus \eta^n[0,t]}(\eta^n(t), R)$ means the renormalized harmonic measure of the right boundary of $\eta^n[0,t]$ union $[0,R]$ seen from infinity. 
See Fig.~\ref{f.RHM1}. 
Let $g^n_t$ be the conformal map from the unbounded connected component of $\HH\setminus\eta^n[0,t]$ onto $\HH$ normalized at $\infty$. 
Now by conformal invariance of RHM, the first term is 
\begin{align*}\label{}
\mathrm{RHM}_{\H \setminus \eta^n[0,t]}(\eta^n(t), R)  & = g^n_t(R) - W^n(t)  \\
& = R + O(\frac {(\diam \eta^n[0,t])^2} R) - W^n(t) \\
& =  R + O(\frac 1 R) - W^n(t)  \text{   by~\eqref{e.diam}} \\
& \to R + O(\frac 1 R) - W(t) \text{   uniformly on }[0,T]\,. 
\end{align*}
Note that we also have 
\begin{align*}\label{}
\mathrm{RHM}_{\H \setminus \eta[0,t]}(\eta(t), R)  & = 
R + O(\frac 1 R) - W(t)\,.
\end{align*}
By letting $R\to \infty$ in the above two displayed equations and using~\eqref{e.thetan}, this concludes our proof modulo the remaining lemma below. 
\end{proof}

\begin{lemma}\label{}
For any $R>M$ (recall~\eqref{e.diam}), 
\begin{align*}\label{}
\mathrm{RHM}_{\H \setminus \eta^n[0,t]}([0,R])   \to  \mathrm{RHM}_{\H \setminus \eta[0,t]}([0,R]) 
\end{align*}
uniformly in $t\in[0,T]$ and the speed of convergence is independent of $R>M$. 
\end{lemma}

\begin{figure}[!htp]
\begin{center}
\includegraphics[width=0.9\textwidth]{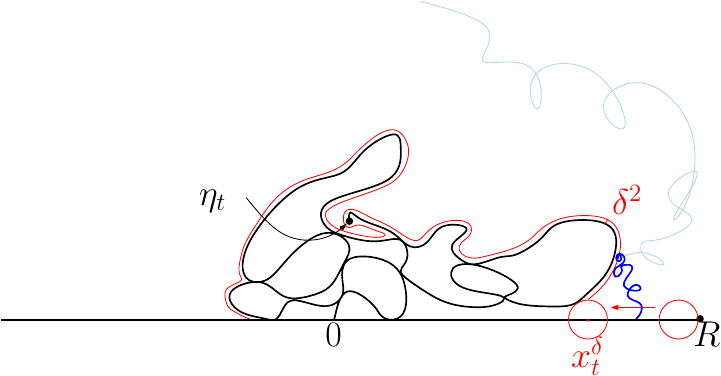}
\end{center}
\caption{The point $x_t^\delta$ can be defined as the center of the $\delta$-ball started at $+\infty$ and shifted towards the origin until it intersects for the first time the curve $\eta([0,t])$. }\label{f.RHM2}
\end{figure}

\begin{proof}
Let $\delta>0$ be some small real number. Define (See Fig.~\ref{f.RHM2})
\[
x_t^\delta := \sup\{ x\in \R_+: \, B(x_t^\delta, \delta) \cap \eta[0,t] \neq \emptyset\}\,,
\]
where $B(x_t^\delta,\delta)$ is the Euclidean ball centred around $x_t^\delta \in \p\H$ of radius $\delta$. 
Let us consider the domains 
\[
\begin{cases}
H_t & := \H \setminus \mathrm{Hull}(\eta[0,t]) \\
H_t^\delta & :=  \H \setminus \bigl( \mathrm{Hull}(\eta[0,t])^{(\delta^2)} \cup B(x_t^\delta, \delta) \bigr)\,,
\end{cases}
\]
where $\mathrm{Hull}(\eta[0,t])^{(\delta^2)}$ is the $\delta^2$-neighborhood in $\H$ of the hull generated by $\eta[0,t]$. See Fig.~\ref{f.RHM2}. 
Clearly, one has $
H_t \subset H_t^\delta$ but we shall not use directly this fact. What we shall use instead is the fact that when $n=n(\delta)$ is sufficiently large then 
\[
H^n_t \subset H_t^\delta
\]
where $H^n_t : = \H \setminus \eta^n[0,t]$. Indeed this follows readily from the fact that $\eta^n \to \eta$ locally uniformly. 
This in turn implies immediately that for $n \geq n(\delta)$, 
\begin{align}\label{}
\mathrm{RHM}_{H_t^\delta}([0,R]) \leq \mathrm{RHM}_{\H \setminus \eta^n[0,t]}([0,R])
\end{align}
Let us now show that there exists some continuous function $f: [0,1] \to [0,1]$ with $f(0)=0$ s.t. 
\begin{align}\label{eqn::RHM_aux}
\mathrm{RHM}_{H_t}([0,R]) \leq \mathrm{RHM}_{H_t^\delta}([0,R])  + f(\delta) 
\end{align}

To show~\eqref{eqn::RHM_aux}, we consider a Brownian motion starting at $\frac 1 u i$ and stopped the first time it hits $\p H_t^\delta$. Let $\tau^\delta$ denote that stopping time. As $H_t \subset H_t^\delta$, one has $\tau^\delta \leq \tau$. Our goal is then to compare $\frac 1 u \FK{}{\frac 1 u i}{B_{\tau} \in [0,R]}$ 
with $\frac 1 u \FK{}{\frac 1 u i}{B_{\tau^\delta} \in [0,R]}$. The difference is given by 
\begin{align*}\label{}
\frac 1 u \FK{}{\frac 1 u i}{B_{\tau^\delta}\in \p \mathrm{Hull}(\eta[0,t])^{(\delta^2)} \setminus B(x_t^\delta, \delta) \text{ and } 
B_\tau \in [0,R]} + 
\frac 1 u \FK{}{\frac 1 u i}{B_{\tau^\delta}\in \p  B(x_t^\delta, \delta)} 
\end{align*}
The second term is less than the RHM of $B(x_t^\delta,\delta)$ in the full $\H$ and is thus bounded from above by $O(\delta)$ as $\delta \to 0$. 
For the first term, notice that
\begin{itemize}
\item $B_{\tau^\delta}$ is at distance $\delta^2$ from $\mathrm{Hull}(\eta[0,t])$ 
\item Because of the definition of $x_t^\delta$ and $B(x_t^\delta, \delta)$, the Brownian motion needs to travel at distance at least $\delta$ between times $\tau^\delta$ and $\tau$ in order to reach $[0,R]$. 
\end{itemize}
Using Beurling's estimate, this happens with probability at most $O(\sqrt{\delta^2/\delta})=O(\delta^{1/2})$. This gives us our desired bound with a continuous function $f$ satisfying $f(\delta)=\Omega(\delta^{1/3})$. 

We have thus shown that 
\begin{align*}\label{}
\mathrm{RHM}_{H_t}([0,R]) & \leq \mathrm{RHM}_{H_t^\delta}([0,R])  + f(\delta)  \\
&  \leq \mathrm{RHM}_{\H \setminus \eta^n[0,t]}([0,R]) + f(\delta)  \text{ for }n\geq n(\delta) 
\end{align*}

Using the exact same proof in the reverse direction (by now defining a point $x_{n,t}^\delta$ which could possibly be very far from $x_t^\delta$), we obtain that 
for $n\geq n(\delta)$, 
\begin{align*}\label{}
\mathrm{RHM}_{\H \setminus \eta^n[0,t]}([0,R]) &  \leq \mathrm{RHM}_{H_t}([0,R]) + f(\delta) 
\end{align*}
which thus concludes the proof. 
\end{proof}

\section{Proof of Theorem~\ref{thm::cvg_interface_chordal}}
\label{sec::proof_main}

\subsection{General setup}
\label{subsec::key_setup}
In this section, we explain the general setup as in Theorem~\ref{thm::cvg_interface_chordal}. 
Suppose $\Omega$ is a bounded simply connected subset of $\C$ with three distinct boundary points (degenerate prime ends) $a, b, c$ in counterclockwise order. Let $(\Omega^{\delta}; a^{\delta}, b^{\delta}, c^{\delta})$ be a sequence of domains on $\delta\Z^2$ converging to $(\Omega; a, b, c)$ in the Carath\'{e}odory sense: 
there exist conformal map $\phi: (\Omega; a, c)\to (\HH; 0, \infty)$ and $\phi^{\delta}: (\Omega^{\delta}; a^{\delta}, c^{\delta}) \to (\HH; 0, \infty)$ so that $\phi^{\delta}\to \phi$ as $\delta\to 0$ uniformly on compact subset of $\Omega$ and $\phi^{\delta}(b^{\delta})\to \phi(b)$.  

Consider the FK-Ising model on $\Omega^{\delta}$ with Dobrushin boundary condition: edges along $\partial_{ba}$ are wired and the dual-edges of $\partial^*_{ab}$ are dual-wired. 
Suppose $\gamma^{\delta}$ is the exploration path from $a^{\delta}$ to $c^{\delta}$, as explained in Section~\ref{subsec::pre_fk}. 
Suppose $\delta_n\to 0$, and denote by $\LA^n=\LA^{\delta_n}$ for $\LA=\Omega, a, b, c, \gamma, \phi, S$ and denote $\eta^n=\phi^{\delta_n}(\gamma^{\delta_n})$. We parameterize $\eta^n$ by the half-plane capacity and denote by $W^n$ the driving function and by $g^n_t$ the corresponding conformal maps in the definition of Loewner chain. Let $\theta^n(t)$ be the renormalized harmonic measure of the right side of $\eta^n[0,t]$ union $[0,\phi^{n}(b^{n})]$ in $\HH\setminus\eta^n[0,t]$ seen from infinity.

\begin{definition}[Definition of stopping times $\{ T_{k}^{n,\eps}, S_{k}^{n,\eps} \}_{k\geq 1}$]\label{d.stop} 
Fix $\eps\ge 10\sqrt{\delta_n}$. Define $T_1^{n,\eps}$ to be the first time that $\theta^n$ is greater than $\eps$ (if $\theta^n(0)=\phi^n(b^n)\ge \eps$, then $T_1^{n, \eps}=0$). Define $S_1^{n,\eps}$ to be the first time after $T_1^{n, \eps}$ that $\gamma^n$ hits the boundary arc $\partial_{bc}$.  
Generally, for $k\ge 1$, let $T_{k+1}^{n,\eps}$ be the first time after $S_k^{n,\eps}$ that $\theta^n$ exceeds $\eps$ and define $S_{k+1}^{n,\eps}$ to be the first time after $T_{k+1}^{n,\eps}$ that $\gamma^n$ hits the boundary arc $\partial_{bc}$.  
\end{definition}

In this way, we decompose the process $\eta^n$ as follows: from time $T_k^{n,\eps}$ to $S_k^{n,\eps}$, the exploration process is similar to the situation when the boundary condition is Dobrushin; from time $S_k^{n,\eps}$ to $T_{k+1}^{n,\eps}$, we know little about the process, and we call this part as \emph{dust}. 

As the sequence $\{\gamma^n\}_n$ satisfies Condition C2 in Definition~\ref{def::c2_chordal} (due to Theorem~\ref{thm::rsw_strong}), from Theorem~\ref{thm::chordal_loewner_cvg}, both sequences $\{\gamma^n\}_n$ and $\{\eta^n\}_n$ are tight. We can extract subsequence, which we still denote by $\{\gamma^n\}_n$ and $\{\eta^n\}_n$, such that 
$W^n$ converges in distribution to $W$ and $\eta^n$ converges in distribution to $\eta$ locally uniformly, and that $\eta$ satisfies the properties in Theorem~\ref{thm::chordal_loewner_cvg}. We couple them on the same probability space so that they converge almost surely. 

For the limiting process $\eta$, recall from Theorem~\ref{thm::chordal_loewner_cvg} that $W$ is its driving function and let $g_t$ be the corresponding conformal maps. Define $\theta(t)$ to be the renormalized harmonic measure of the right boundary of $\eta[0,t]$ union $[0,\phi(b)]$ seen from infinity. Fix $T>0$ large and define 
\[\|\LA^n-\LA\|_{\infty, T}:=\sup\{|\LA^n(t)-\LA(t)|: t\in [0,T]\}, \]
for $\LA=\eta, W, \theta$. 
By Proposition~\ref{pr.thetan}, we know that $\|\theta^n-\theta\|_{\infty, T}\to 0$ almost surely. 

For the limiting process $\eta$, we define the stopping times similarly. Let $T_1^{\eps}$ be the first time that $\theta$ is greater than $\eps$. Define $S_1^{\eps}$ to be the first time after $T_1^{\eps}$ that $\theta$ hits zero. Generally, for $k\ge 1$, define $T_{k+1}^{\eps}$ to be the first time after $S_k^{\eps}$ that $\theta$ exceeds $\eps$ and define $S_{k+1}^{\eps}$ to be the first time after $T_{k+1}^{\eps}$ that $\theta$ hits zero. 

The goal of Theorem~\ref{thm::cvg_interface_chordal} is to identify the law of $\eta$ and our strategy is the following: 
\begin{itemize}
\item First we argue that 
the following two processes are close: 
\[\left(\theta^n(t), T_k^{n,\eps}\le t\le S_k^{n,\eps}\right)\quad\text{and}\quad \left(\theta(t), T_k^{\eps}\le t\le S_k^{\eps}\right).\]
\item Then, using Theorem~\ref{thm::fkising_cvg_Dobrushin}, we shall argue that $\left(\theta^n(t), T_k^{n,\eps}\le t\le S_k^{n,\eps}\right)$ converges in distribution to Bessel excursion, and thus $\left(\theta(t), T_k^{\eps}\le t\le S_k^{\eps}\right)$ is a Bessel excursion.

\item Using Lemma~\ref{lem::approximate_bessel} and Lemma~\ref{lem::chordal_evolution_zeroLeb}, we conclude that $\theta$ is a Bessel process.

\item Finally, we use Lemma~\ref{lem::chordal_evolution_boundary_point} to extract $W$ from $\theta$ and conclude that $\eta$ is an $\SLE_{\kappa}(\kappa-6)$. 
\end{itemize}
 
However, it is quite delicate to make this strategy work. The first issue is that, although the processes $(\eta^n, W^n, \theta^n)$ are close to $(\eta, W, \theta)$, we do not know a priori whether the stopping times $(T_k^{n,\eps}, S_k^{n,\eps})$ are close to the stopping times $(T_k^{\eps}, S_k^{\eps})$. This will be proved in Proposition~\ref{pr.ST} and this turns out to be slightly more technical than one might expect. 

For the second item, the issue is that one needs to argue the processes are not moving much for $\theta^n$ and $\theta$ on 
\[
[T_k^{n,\eps} \wedge T_k^{\eps} , T_k^{n,\eps}  \vee  T_k^{\eps}].
\]
This issue will be solved by equicontinuity, see Section~\ref{subsec::cvg_oneBesselexcursion}.

Another issue concerns the convergence of conditional distribution, or the passage of Markov property to the limit. 
In the discrete, the exploration process $\eta^n$ has domain Markov property 
 and we know $\eta^n$ converges to $\eta$. But the domain Markov property does not pass to $\eta$ for free, as it was pointed out in \cite[Proposition~4.2 and Section~5]{SchrammFirstSLE} in the setting of loop-erased random walk. 
For simplicity, we first discuss the following two pieces
\[X^n:=(\eta^n(t), 0\le t\le T_1^{n,\eps})\quad\text{and}\quad Y^n:=(\eta^n(t), T_1^{n,\eps}\le t\le S_1^{n,\eps}).\] 

Define the conformal map $G^n(\cdot):=g^n_{T_1^{n,\eps}}(\cdot)-W^n(T_1^{n,\eps})$. Note that $G^n$ is a measurable function of $X^n$. In the limiting process $\eta$, we define
\[X:=(\eta(t), 0\le t\le T_1^{\eps})\quad\text{and}\quad Y:=(\eta(t), T_1^{\eps}\le t \le S_1^{\eps}).\]
Define the conformal map $G(\cdot):=g_{T_1^{\eps}}(\cdot)-W(T_1^{\eps})$ and note again that $G$ is a measurable function of $X$. 
At this point, we have $\eta^n\to \eta$ and $S^{n,\eps}_1\to S_1^{\eps}$, and hence we have
the convergence of the concatenation of $(X^n, Y^n)$ to the concatenation of $(X, Y)$ in the metric~\eqref{eqn::metric_space_curves},
and we want to argue that the conditional law of $Y_n$ given $X_n$ converges to the conditional law of $Y$ given $X$. 
However, this is generally false without further condition on $(X^n, Y^n)$, see for example the discussion in \cite{EimearConvergenceDistributionConditionalExpectations}. 

In our setting, we do have further properties below on the pair $(X^n, Y^n)$ which allow us to conclude. 
\begin{itemize}
\item As $(\eta^n, W^n, \theta^n)$ converges to $(\eta, W, \theta)$ and $T_1^{n,\eps}\to T_1^{\eps}$ almost surely, we see $G^n\to G$ in Carath\'{e}odory sense. (This follows for example from Caratheodory kernel theorem).
As $\eta^n\to \eta$, together with equicontinuity and the fact that $T_1^{n,\eps}\to T_1^{\eps}$ a.s. we obtain that $X^n$ converges to $X$.
Consider the $G^n\left(\eta^n|_{t\ge T_1^{n,\eps}}\right)$. The collection 
$\left\{G^n\left(\eta^n|_{t\ge T_1^{n,\eps}}\right)\right\}_n$ satisfies Condition C2 due to Theorem~\ref{thm::rsw_strong}, hence it is tight in the topologies in Theorem~\ref{thm::chordal_loewner_cvg}. Combining with the equicontinuity and $S_1^{n,\eps}\to S_1^{\eps}$, we will show in Lemma~\ref{lem::cvg_GnYn} that $G(Y)$ is the only possible subsequential limit, thus $G^n(Y^n)$ converges to $G(Y)$.

\item Since $Y^n$ is an exploration path in $\HH\setminus X^n$ with Dobrushin boundary condition (stopped at the disconnection time), using Theorem~\ref{thm::fkising_cvg_Dobrushin} and Lemma~\ref{lem::sle_targetindep}, we will obtain in Section~\ref{subsec::cvg_oneBesselexcursion}
that $G^n(Y^n)$ converges to an $\SLE_{\kappa}(\kappa-6)$ in $\HH$ from $0$ to $\infty$ with force point at $\eps$ (stopped at the disconnection time). 
This will be the purpose of Proposition~\ref{prop::cvg_oneBesselexcursion} and the key point there will be that 
$G(Y)$ is independent of $X$. 
\end{itemize}

Going back to the random process $\theta$, we will see in Section~\ref{subsec::cvg_oneBesselexcursion} that by combining these three observations,  one obtains: $(\theta(t), 0\le t\le T_1^{\eps})$ and 
$(\theta(t), T_1^{\eps}\le t\le S_1^{\eps})$ are independent and $(\theta(t), T_1^{\eps}\le t\le S_1^{\eps})$ has the law of Bessel excursion.

For general $k\ge 1$, the above argument applies to the following two pieces 
\[(\eta^n(t), 0\le t\le T_k^{n,\eps})\quad\text{and}\quad (\eta^n(t), T_k^{n,\eps}\le t\le S_k^{n,\eps}).\]
As hinted above, Section~\ref{subsec::cvg_besselprocess} will combine the above analysis together with Lemmas~\ref{lem::approximate_bessel} and~\ref{lem::chordal_evolution_zeroLeb} in order to conclude that $\theta$ is a Bessel process.  

\subsection{Convergence of discrete stopping times to their continuous analogs}\label{ss.ST}
In this section we shall prove the following key result on $\{T_{k}^{n,\eps}, T_k^\eps, S_{k}^{n,\eps}, S_{k}^\eps\}$
\begin{proposition}\label{pr.ST}
Assume we are in the above setup where, in particular, $\eta^n \to \eta$ locally uniformly and $W^n \to W$ locally uniformly.
Then we have for any $k\geq 1$ and as $n\to \infty$, 
\begin{align*}\label{}
\begin{cases}
& T_{k}^{n,\eps} \to T_k^\eps \text{ in probability} \\
& S_{k}^{n,\eps} \to S_k^\eps \text{ in probability} 
\end{cases}
\end{align*}

We shall in fact need the following slightly more precise version. 
For any fixed $T>0$, there exists a sequence $\{\alpha_n\}_n$ converging to zero such that the following holds:
\begin{align*}\label{}
\Pb{\exists k\geq 1,\text{ s.t. } T_k^\eps \leq T -1  \text{ and } |T_{k}^{n,\eps}-T_{k}^\eps| > \alpha_n} \leq \alpha_n  \\
\Pb{\exists k\geq 1,\text{ s.t. } S_k^\eps \leq T -1  \text{ and } |S_{k}^{n,\eps}-S_{k}^\eps| > \alpha_n} \leq \alpha_n 
\end{align*}
\end{proposition}

\begin{remark}\label{rem::ruiningMarkovproperty}
This is in the same flavour as \cite[Lemma 3.1]{WernerLecturePercolation} which would correspond to $T_k^{n,\eps} \to T_k^\eps$ a.s., the proof of which was left as a homework exercise. However, we find this exercise not that easy for the following reasons:
\begin{enumerate}
\item[(1)] First, if we stop the joint exploration paths $(\eta^n,\tilde \eta)$ at time $S$, indeed $\eta^n$ is close to disconnecting and it is tempting to conclude by some careful use of RSW. But one important issue is that $S$ is a stopping time for $\tilde \eta$ but not for $\eta^n$. Because of that, we are not allowed to use the discrete domain Markov property and a rather delicate analysis cannot be avoided it seems. 
\item[(2)] Second, stopping the curve when it is close to disconnecting $0$ from $\p \U$ needs to be done with some care. For example, being close in $\| \cdot\|_\infty$ to disconnection does not prevent from having a dual harmonic arc with large harmonic measure seen from $0$. This is why in our proof below, we rely on stopping times built from harmonic measure seen from $0$ instead of Euclidean distance from disconnection. 
\end{enumerate}
\end{remark}

Let us start as a warm-up with the following Lemma.
\begin{lemma}\label{l.FirstST}
Assume for simplicity that we are in the case where $\theta^n(0)=\theta(0)=0$ (no free arc at the beginning of the exploration). For any fixed $T\geq 2$,  We have for any $u>0$,  
\begin{align*}
\Pb{T_1^\eps \leq T-1 \text{ and } |T_1^{n,\eps} - T_1^\eps|>u} \to 0\,, \text{ as $n\to \infty$.}
\end{align*}

\end{lemma}

For any $r<\eps/2$, define the stopping times $T_1^{n, \eps-2r}$ and $T_{1}^{n, \eps+ 2r}$ exactly as $T_{1}^{n,\eps}$. By definition and monotony, one clearly has
\begin{align}\label{e.mon}
T_1^{n,\eps- 2r} \leq T_1^{n,\eps} \leq T_1^{n, \eps+ 2r}
\end{align}
Now, let us show that there exists a function $f(r)$ which goes to zero as $r\to 0$, and which is such that uniformly in $n\geq n(r)$, one has 
\begin{align}\label{e.2r}
\Pb{T_1^{n, \eps+ 2r} - T_1^{n,\eps- 2r} \geq f(r)} \leq f(r)
\end{align}

\begin{remark}\label{}
Recall the main issue in the current proof is that any interaction between $\eta$ and $\eta^n$ may ruin the domain Markov property for $\eta^n$. 
The above estimate~\eqref{e.2r} does not involve the limiting curve $\eta$ in the joint coupling and it is therefore much safer to prove such an estimate  using standard arguments. 
\end{remark}

\begin{figure}[!htp]
\begin{center}
\includegraphics[width=0.9\textwidth]{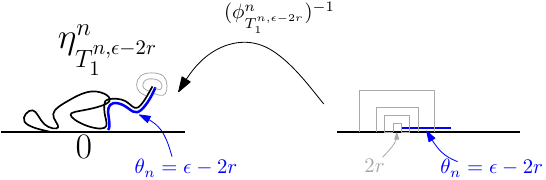}
\end{center}
\caption{Let $g^n_t$ be the conformal map from $\HH\setminus\eta^n[0,t]$ onto $\HH$ normalized at $\infty$. Define $\phi^n_t=g^n_t-W^n_t$.}\label{f.2r}
\end{figure}

\begin{proof}[Proof of Lemma~\ref{l.FirstST}]
We first prove~\eqref{e.2r}, we use the strong RSW in appropriate quads in the discrete domain $\H \setminus \eta^n[0, T_1^{n,\eps-2r}]$ which are defined as conformal images via $(\phi^n_{T_1^{n,\eps-2r}})^{-1}$ of well-chosen rectangular quads in $\H$, see Fig.~\ref{f.2r}.  Namely, denote by $K=[-r,r]\times[0,r]$ and we draw the quads $A_1=(2K)\setminus K$, and $A_j=2^j A_1$ for $j\ge 1$. From Lemma~\ref{l.tech}, we know that the extremal length of the discrete quad approximation of $(\phi^n_{T_1^{n,\eps-2r}})^{-1}(A_j)$ is bounded by universal constants. Denote by $\LE_j$ the event that there is a dual crossing in the discrete quad approximation of $(\phi^n_{T_1^{n,\eps-2r}})^{-1}(A_j)$. Then Theorem~\ref{thm::rsw_strong} gives $\PP[\LE_j]\ge c_0$ for some universal constant $c_0>0$. If $\LE_j$ holds, then 
\[T_1^{n,\eps+2r}\le T_1^{n,\eps-2r}+4^jr^2.\]
Therefore,
\[\PP[T_1^{n,\eps+2r}-T_1^{n,\eps-2r}\ge 4^j r^2]\le (1-c_0)^j. \]
This gives~\eqref{e.2r}. 

Now, using that $\eta^n \to \eta$ locally uniformly and $W^n \to W$ locally uniformly, recall we have by Proposition \ref{pr.thetan} that 
$\| \theta^n -\theta \|_{\infty, T} \to 0$ as $n\to \infty$. In particular, if we define the event 
\begin{align*}\label{}
E^{n,r}:= \{\| \theta^n -\theta \|_{\infty, T}  \leq r \}\,,
\end{align*}
then we have for any $r>0$, $\Pb{E^{n,r}} \to 1$ as $n\to \infty$. 
The main observation which remains in order to prove Lemma \ref{l.FirstST} is that on the event 
\[
 \{\| \theta^n -\theta \|_{\infty, T}  \leq r \}
\]
we must have the inequality
\[
T_1^{n,\eps-2r} \leq T_1^\eps \leq T_1^{n,\eps+2r}
\]
at least if $T_1^{n,\eps+2r}$ is not too big (it needs to be less than $T$ which it does with high probability on the event $T_1^\eps \leq T-1$ thanks to the estimate ~\eqref{e.2r}). 
This together with~\eqref{e.mon} readily implies that on the event $ \{\| \theta^n -\theta \|_{\infty, T}  \leq r \} \cap \{ T_1^{n, \eps+ 2r} - T_1^{n,\eps- 2r} < f(r) \}$, one has $|T_1^{n,\eps}-T_1^{\eps}| < f(r)$ if $T_{1}^\eps \leq T-1$. As $\liminf \Pb{E^{n,r} \cap \{ T_1^{n, \eps+ 2r} - T_1^{n,\eps- 2r} < f(r) \}} \geq 1 -f(r)$, this concludes the proof of Lemma \ref{l.FirstST} by choosing $r$ arbitrarily small. 
\end{proof}

In order to prove Proposition \ref{pr.ST}, we would like to iterate the same idea to the later stopping times $S_{k}^{n,\eps}, T_{k}^{n,\eps}$ etc.

\begin{proof}[Proof of Proposition \ref{pr.ST}]
Let us start by explaining in details how to handle the convergence of the next stopping time, i.e. $S_1^{n,\eps} \overset{\text{Prob.}}\to S_1^{\eps}$. Namely we wish to prove that for any $u>0$, 
\begin{align}\label{e.mon2}
\Pb{S_1^\eps \leq T-1 \text{ and } |S_1^{n,\eps} - S_1^\eps|>u} \to 0\,, \text{ as $n\to \infty$.}
\end{align}
To prove this, we face two (slight) technical difficulties:
\begin{enumerate}
\item[(1)] The first one is that $S_1^{n,\eps}$ will be close to $S_1^\eps$ only if the earlier stopping times $T_1^{n,\eps}$ and $T_1^\eps$ will be close as well. This must appear in the proof somewhere. 
\item[(2)] The second issue is that there is no monotonicity such as the one we used above (namely, $T_1^{n,\eps-2r} \leq T_1^\eps \leq T_1^{n,\eps+2r}$). We will still have an analog of the left inequality, but the R.H.S will be replaced by the inequality $S_1^\eps \leq \liminf_{n\to \infty} S_1^{n,\eps}$ which can be seen as a deterministic statement given the fact that $\theta^n \to \theta$ uniformly on $[0,T]$. 
\end{enumerate}

Let us introduce the following stopping times which will have useful monotony properties:
\begin{align*}\label{}
\begin{cases}
& \tilde S_1^{n,\eps,2r}:= \inf\{ t > T_1^{n,\eps- 2r}, \text{s.t. } \theta^n(t) \leq 2r \} \\
& \hat S_1^{n,\eps,2r}:= \inf\{ t > T_1^{n,\eps- 2r}, \text{s.t. $\gamma^n$ hits the boundary arc $\partial_{bc}$} \}
\end{cases}
\end{align*}
Note first that it always the case that  
\begin{align}\label{e.monoS1}
\tilde S_1^{n,\eps,2r} \leq S_1^{n,\eps}
\end{align}
Also, note that on the event $\{ S_1^\eps \leq T-1\} \cap \{ \| \theta^n -\theta \|_{\infty, T}  \leq r \}$, we have that
\begin{align}\label{e.monoS2}
\tilde S_1^{n,\eps,2r} \leq S_1^{\eps}
\end{align}
Furthermore, exactly as for the estimate~\eqref{e.2r}, one can prove in the same fashion that there exists a function $f(r)$ which goes to zero as $r\to 0$, and which is such that uniformly in $n\geq n(r)$, one has 
\begin{align}\label{e.hat2r}
\Pb{\hat S_1^{n, \eps,2r} - \tilde S_1^{n,\eps,2r} \geq f(r)} \leq f(r)  
\end{align}
Finally, using the estimate~\eqref{e.2r} as well as the equicontinuity of the set of functions $\{ \theta^n \}_n$ restricted to the interval $[0,T]$ (this equicontinuity follows form the uniform convergence of $\theta^n$ towards the continuous $\theta$), we deduce that for $n$ large enough, 
\begin{align}\label{e.equic}
\Pb{\hat S_1^{n,\eps, 2r} = S_1^{n,\eps}} \geq 1 - 2 \, f(r)\,.
\end{align}
Indeed, one term $f(r)$ comes from the possibility that $T_1^{n, \eps+ 2r} \gg T_1^{n,\eps-2r}$ which could prevent the above equality to hold and is dominated thanks to~\eqref{e.2r}, the second term $f(r)$ comes from the unlikely event that $\gamma^n$ would hit the arc $\p_{bc}$ strictly between $T_1^{n, \eps- 2r}$ and  $T_1^{n,\eps}$. This possibility is easily controlled using the equicontinuity of $\{\theta^n\}_n$.  
Combining the above four estimates \eqref{e.monoS1}, \eqref{e.monoS2}, \eqref{e.hat2r}, \eqref{e.equic}, we obtain that for any $u>0$, 
\begin{align*}\label{}
\Pb{S_1^{\eps} - S_1^{n,\eps} < -u} \to 0 \text{  as $n\to \infty$.}
\end{align*}

For the other direction, we rely on a completely different argument (already suggested in \cite[Lemma 3.1]{WernerLecturePercolation}) which is based on the Lemma stated below. Indeed it readily implies that 
\begin{align*}\label{}
\Pb{S_1^{\eps} - S_1^{n,\eps} > +u} \to 0 \text{  as $n\to \infty$.}
\end{align*}
which concludes our proof at least for $S_1^{n,\eps}\overset{\text{Prob.}}\to S_1^{\eps}$. 
\end{proof}

\begin{lemma}\label{l.limsup}
On the event $S_1^\eps \leq T-1$, 
\[
S_1^\eps \leq \liminf S_1^{n,\eps}
\]
\end{lemma}

\begin{proof}
Let us argue by contradiction. Suppose this is not the case, then there exists $\alpha>0$ such that for infinitely many $n_k\in \N$, 
\[
S_1^{n_k,\eps} \leq S_1^\eps - \alpha\,.
\]
Using Beurling's estimate one has
\begin{align*}\label{}
\begin{cases}
&|\theta^{n_k}(T_1^{n_k,\eps}) - \eps| \leq 10\sqrt{\delta_{n_k}} \\
&\theta^{n_k}(S_1^{n_k,\eps})\le 10\sqrt{\delta_{n_k}}
\end{cases}
\end{align*}
By possibly further extracting so that $T_1^{n_k,\eps}$ and $S_1^{n_k,\eps}$ both converge and using the fact that $\theta^{n_k}$ converges uniformly to $\theta$ on $[0,T]$, we thus reach a contradiction, as $S_1^\eps$ should then be smaller than $\lim S_1^{n_k,\eps}$. 
\end{proof}

\begin{proof}[Proof of Proposition \ref{pr.ST} continued]
For the general case, $k\geq 2$, we can proceed inductively on $k\geq 1$. The induction hypothesis being that indeed, $T_j^{n,\eps}\overset{\text{Prob.}}\to T_j^{\eps}$ and $S_j^{n,\eps}\overset{\text{Prob.}}\to S_j^{\eps}$ for all $j \leq k-1$. Then, to propagate the induction hypothesis, we proceed as follows: say we have proved all stopping times converge in probability all the way to $T_k^{n,\eps}\overset{\text{Prob.}}\to T_k^{\eps}$ and we wish to control the next one, i.e. $S_k^{n,\eps}\overset{\text{Prob.}}\to S_k^{\eps}$. For the lower bound, we set up the following stopping times:
\begin{align*}\label{}
\begin{cases}
& \tilde T_2^{n,\eps,2r}:= \inf\{ t> \tilde S_1^{n,\eps,2r}, \theta^n(t)=\eps-2r\} \\
& \hat T_2^{n,\eps,2r}:= \inf\{ t> \tilde S_1^{n,\eps,2r}, \theta^n(t)=\eps+2r\} \\
& \tilde S_2^{n,\eps,2r}:= \inf\{ t> \tilde T_2^{n,\eps,2r}, \theta^n(t)= 2r\} \\ 
& \hat S_2^{n,\eps,2r}:= \inf \{ t>\tilde T_2^{n,\eps,2r},  \text{s.t. $\gamma^n$ hits the boundary arc $\partial_{bc}$}\} \\
& \tilde T_3^{n,\eps,2r}:= \inf\{ t> \tilde S_2^{n,\eps,2r}, \theta^n(t)=\eps-2r\} \\
& \ldots
\end{cases}
\end{align*}
The advantage of these definitions is that stopping times $\tilde T_k$ and $\hat S_k$ (resp. $\tilde S_k$ and $\hat S_k$) are close with high probability, and the following monotonies always hold:
\begin{align*}\label{}
& \tilde T_k^{n,\eps,2r}:= \inf\{ t> \tilde S_{k-1}^{n,\eps,2r}, \theta^n(t)=\eps-2r\}  \leq T_k^\eps \\
& \tilde S_k^{n,\eps,2r}:= \inf\{ t> \tilde T_k^{n,\eps,2r}, \theta^n(t)=\eps-2r\}  \leq S_k^\eps .
\end{align*}

The same proof as the one above implies that for all $k\geq 2$ and $u>0$, 
\begin{align*}\label{}
\Pb{T_k^{\eps} - T_k^{n,\eps} < -u}\vee \Pb{S_k^{\eps} - S_k^{n,\eps} < -u}  \to 0 \text{  as $n\to \infty$.}
\end{align*}

Now, for the upper bound, exactly as in Lemma \ref{l.limsup}, one has for all $k\geq 2$, 
\[ T_k^\eps \leq \liminf T_k^{n,\eps} ,\quad S_k^\eps \leq \liminf S_k^{n,\eps}\]
which concludes the proof that one can iterate from $j\leq k-1$ to $k$ in the same way as for $S_1^{n,\eps} \overset{\text{Prob.}}\to S_1^{\eps}$ above. 

To conclude our proof of Proposition \ref{pr.ST}, one still need to handle a  potentially large number of stopping times. Indeed the main statement in Proposition \ref{pr.ST} provides a control on ALL stopping times $T_k^{n,\eps}$ or $S_k^{n,\eps}$ which arise below $T$. To conclude, we thus rely once again on the equicontinuity of $\{\theta^n\}_n$ on $[0,T]$ (which again follows from $\theta^n \to \theta$ uniformly on $[0,T]$). In particular, there is a random $\delta=\delta(\omega,\eps)>0$ a.s., s.t. for all $n\geq 1$ and any $0\leq s < t \leq T$, with $|s-t| < \delta$
\[
|\theta^n(s) - \theta^n(t)| < \eps/2\,.
\]
This implies readily that one cannot have more than $T \delta^{-1}$ stopping times before time $T$. Now by combining the fact that $\Pb{\delta(\omega,\eps) > \alpha}\to 1$ as $\alpha \searrow 0$ and a straightforward union bound argument, we conclude the proof of Proposition \ref{pr.ST} with a choice of $\{\alpha_n\}_n$ converging sufficiently slowly to zero. 
\end{proof}

\subsection{Convergence in law to one Bessel excursion}
\label{subsec::cvg_oneBesselexcursion}

In this section, we will show the following proposition.   
\begin{proposition}\label{prop::cvg_oneBesselexcursion}
The law of $(\theta(t)/\sqrt{\kappa}, T_1^{\eps}\le t\le S_1^{\eps})$ is the same as  a Bessel process of dimension $3-8/\kappa$ starting from $\eps$ and stopped when it reaches zero where $\kappa=16/3$. Moreover, it is independent of $(\theta(t), t\le T_1^{\eps})$. 
\end{proposition}
We will give a detailed proof of Proposition~\ref{prop::cvg_oneBesselexcursion} in this section, and most of the arguments can be applied verbatim for the future excursions.

From Proposition~\ref{pr.ST}, we have
\[
\PP\left[S_1^{\eps}\le T-1, |T_1^{n,\eps}-T_1^{\eps}|>\alpha_n, |S_1^{n,\eps}-S_1^{\eps}|>\alpha_n\right]\le \alpha_n.
\]
We may choose a subsequence $n_j\to\infty$ such that $\sum_j\alpha_{n_j}<\infty$. Then we have
\[\sum_j\PP\left[S_1^{\eps}\le T-1, |T_1^{n_j,\eps}-T_1^{\eps}|>\alpha_{n_j}, |S_1^{n_j,\eps}-S_1^{\eps}|>\alpha_{n_j}\right]<\infty.\]
By Borel-Cantelli Lemma, we have 
\[
T_1^{n_j,\eps}\to T_1^{\eps},\quad S_1^{n_j, \eps}\to S_1^{\eps},\quad\text{a.s. on }\{S_1^{\eps}\le T-1\}.
\]

\begin{lemma}\label{lem::cvg_GnYn}
Recall the definition of $X^n, Y^n, G^n$ and $X, Y, G$ as in Section~\ref{subsec::key_setup}. On the event $\{S_1^{\eps}\le T-1\}$, the process $G^{n_j}(Y^{n_j})$ converges to $G(Y)$ almost surely. 
\end{lemma}

\begin{proof}
Consider the sequence $\left\{G^{n_j}\left(\eta^{n_j}|_{t\ge T_1^{n_j,\eps}}\right)\right\}_j$, it satisfies Condition C2 by Theorem~\ref{thm::rsw_strong}. Then it is tight as in Theorem~\ref{thm::chordal_loewner_cvg}. Suppose $\left\{G^{n_{k_j}}\left(\eta^{n_{k_j}}|_{t\ge T_1^{n_{k_j},\eps}}\right)\right\}_j$ is a converging subsequence and the limit is $\tilde{\eta}$ with driving function $\tilde{W}$. We have the following observation.
\begin{itemize}
\item Applying Theorem~\ref{thm::chordal_loewner_cvg} to $\left\{G^{n_{k_j}}\left(\eta^{n_{k_j}}|_{t\ge T_1^{n_{k_j},\eps}}\right)\right\}_j$, we know that $\left\{W^{n_{k_j}}\right\}_j$, restricted to $\left[T_1^{n_{k_j}, \eps}, T\right]$, converges to $\tilde{W}$ locally uniformly. 
\item Applying Theorem~\ref{thm::chordal_loewner_cvg} to $\eta^n$, we know that $W^n$ converges to $W$ locally uniformly. In particular, $W^n$ converges to $W$ uniformly on $[0,T]$. The uniform convergence implies the equicontinuity of the sequence $\{W^n\}_n$ on the interval $[0,T]$. 
\item By the choice of $n_j$, we have $T_1^{n_j,\eps}\to T_1, S_1^{n_j,\eps}\to S_1^{\eps}$ almost surely on $\{S_1^{\eps}\le T-1\}$.
\end{itemize}
Combining the above three facts, we conclude that $\tilde{W}$ coincides with $W$ on the interval $[T_1^{\eps}, S_1^{\eps}]$. 
By Theorem~\ref{thm::chordal_loewner_cvg} again, $\tilde{\eta}$ is the Loewner chain generated by $(\tilde{W}(t), T_1^{\eps}\le t\le S_1^{\eps})$ and $G(Y)$ is the Loewner chain generated by $(W(t), T_1^{\eps}\le t\le S_1^{\eps})$. Thus $\tilde{\eta}$ coincides with $G(Y)$. As $G(Y)$ is the unique subsequential limit of $\{G^{n_j}(Y^{n_j})\}_j$, we conclude that $G^{n_j}(Y^{n_j})$ converges to $G(Y)$. 
\end{proof}

\begin{proof}[Proof of Proposition~\ref{prop::cvg_oneBesselexcursion}]
Recall the definition of $X^n, Y^n, G^n$ and $X, Y, G$ as in Section~\ref{subsec::key_setup}.

First, we explain the convergence of $G^{n_j}$. 
\begin{itemize}
\item Recall that $G^{n_j}$ is the conformal map from $(\HH\setminus\eta^{n_j}[0, T_1^{n_j,\eps}]; \eta^{n_j}(T_1^{n_j,\eps}),\infty)$ onto $(\HH; 0,\infty)$ and $G$ is the conformal map from $(\HH\setminus\eta[0,T_1^{\eps}]; \eta(T_1^{\eps}), \infty)$ (normalized at $\infty$). As $\eta^n\to \eta, W^n\to W$ and $T_1^{n_j,\eps}\to T_1^{\eps}$, The map $G^{n_j}\to G$ locally uniformly almost surely on $\{S_1^{\eps}\le T-1\}$ by Carath\'{e}odory kernel theorem.  
\item By Proposition~\ref{pr.thetan}, we have $\theta^n\to \theta$ locally uniformly. In particular, the sequence $\{\theta^n\}_n$ is equicontinuous on $[0,T]$. 
As $T_1^{n_j,\eps}\to T_1^{\eps}$, we conclude $\theta^{n_j}(T_1^{n_j,\eps})\to \theta(T_1^{\eps})=\eps$ almost surely on $\{S_1^{\eps}\le T-1\}$. 
\end{itemize}
Combining the above two items, we conclude that $G^{n_j}\to G$ in Carath\'{e}odory sense and the image of the dual arc under $G^{n_j}$ converges to the interval $[0,\eps]$ almost surely on $\{S_1^{\eps}\le T-1\}$. 

Since $Y^n$ is the exploration path in $\HH\setminus X^n$ with Dobrushin boundary condition. Combining the above convergence of $G^{n_j}$ and Theorem~\ref{thm::fkising_cvg_Dobrushin}, the sequence $G^{n_j}(Y^{n_j})$ converges in distribution to $\SLE_{\kappa}$ in $\HH$ from $0$ to $\eps$. By the choice of $n_j$, we also have the convergence of the disconnection time: $S_1^{n_j,\eps}\to S_1^{\eps}$ almost surely on $\{S_1^{\eps}\le T-1\}$. Therefore, $G^{n_j}(Y^{n_j})$ up to the disconnection time converges in distribution to $\SLE_{\kappa}$ in $\HH$ from $0$ to $\eps$ up to the disconnection time. By Lemma~\ref{lem::sle_targetindep}, we conclude that  $G^{n_j}(Y^{n_j})$ (up to the disconnection time) converges in distribution to $\SLE_{\kappa}(\kappa-6)$ (up to the disconnection time) in $\HH$ from $0$ to $\infty$ with force point $\eps$. 

By Lemma~\ref{lem::cvg_GnYn}, we have $G^{n_j}(Y^{n_j})\to G(Y)$ almost surely. Combining with the above analysis, we know that $G(Y)$ has the law of  $\SLE_{\kappa}(\kappa-6)$ (up to the disconnection time) in $\HH$ from $0$ to $\infty$ with force point $\eps$. As $\theta$ is the corresponding renormalized harmonic measure, it has the same law as Bessel process starting from $\eps$ stopped at the first time that it reaches zero conditioned on $\{S_1^{\eps}\le T-1\}$. As this is true for all $T\ge 2$, and $\PP[S_1^{\eps}\le T-1]\to 1$ as $T\to \infty$, we can remove the conditioning.

It thus remains to argue that $G(Y)$ is indeed independent of $X$. Let us show that for any bounded continuous functionals $f$ and $h$ on the space of continuous curves with the topology of local uniform convergence,
one has $\Eb{f(G(Y)) h(X)} = \Eb{f(G(Y))} \Eb{h(X)}$. As we have shown above that $X^{n_j} \to X$ a.s. and (Lemma \ref{lem::cvg_GnYn}) that $G^{n_j}(Y^{n_j}) \to G(Y)$ a.s., we readily have by dominated convergence theorem that 
\begin{align*}\label{}
\Eb{f(G(Y)) h(X)} & = \lim_{j \to \infty} \Eb{f(G^{n_j}(Y^{n_j})) h(X^{n_j})} \\
& = \lim_{j \to \infty} \Eb{ \Eb{f(G^{n_j}(Y^{n_j})) \md X^{n_j}} h(X^{n_j})}
\end{align*}
Now, observe that the above analysis in fact shows more than what we stated. Namely, on the event $K$ that $G^{n_j}\to G$ and $\theta^{n_j}(T_1^{n_j,\eps})\to \theta(T_1^{\eps})=\eps$ (which happens with probability one as shown above), we have as argued above by Theorem~\ref{thm::fkising_cvg_Dobrushin} and Lemma~\ref{lem::sle_targetindep} that the law of $G^{n_j}(Y^{n_j})$ (up to the disconnection time) \underline{given $X^{n_j}$} converges in distribution to $\SLE_{\kappa}(\kappa-6)$ (up to the disconnection time) in $\HH$ from $0$ to $\infty$ with force point $\eps$. This is nothing but saying that for any functional $f$ as above, one has $\Eb{f(G^{n_j}(Y^{n_j})) \md X^{n^j}} \to \Eb{f(\SLE_{\kappa}(\kappa-6))}$ almost surely on the event $K$. 
As $f$ is bounded and $\Pb{K}=1$, again by dominated convergence theorem, one has 
\begin{align*}\label{}
\Eb{f(G(Y)) h(X)} & = \lim_{j \to \infty} \Eb{ \Eb{f(G^{n_j}(Y^{n_j})) \md X^{n_j}} h(X^{n_j})} \\
& = \Eb{ \Eb{f(\SLE_{\kappa}(\kappa-6))}  h(X)} \\
&= \Eb{f(\SLE_{\kappa}(\kappa-6))} \Eb{h(X)} = \Eb{f(G(Y))} \Eb{h(X)}
\end{align*}
which thus concludes the proof. 
\end{proof}

For general $k\ge 1$, by Proposition~\ref{pr.ST}, we have 
\[\PP[S_k^{\eps}\le T-1, \exists \ell\le k \text{ s.t. }|T_{\ell}^{n,\eps}-T_{\ell}^{\eps}|>\alpha_n \text{ or }|S_{\ell}^{n,\eps}-S_{\ell}^{\eps}|>\alpha_n]\le \alpha_n.\]
As $\sum_j\alpha_{n_j}<\infty$, we have 
\[\sum_j\PP[S_k^{\eps}\le T-1, \exists \ell\le k \text{ s.t. }|T_{\ell}^{n_j,\eps}-T_{\ell}^{\eps}|>\alpha_{n_j} \text{ or }|S_{\ell}^{n_j,\eps}-S_{\ell}^{\eps}|>\alpha_{n_j}]<\infty.\]
By the Borel-Cantelli lemma, we have 
\[T_{\ell}^{n_j,\eps}\to T_{\ell}^{\eps}, \quad S_{\ell}^{n_j, \eps}\to S_{\ell}^{n_j,\eps},\quad\text{for all }\ell\le k,\quad\text{a.s. on }\{S_k^{\eps}\le T-1\}.\]
The proof of Proposition~\ref{prop::cvg_oneBesselexcursion} also works for $(\theta(t), T_k^{\eps}\le t\le S_k^{\eps})$.

\begin{corollary}\label{cor::cvg_oneBesselexcursion}
For any $k\ge 1$, the law of $(\theta(t)/\sqrt{\kappa}, T_k^{\eps}\le t\le S_k^{\eps})$ is the same as a Bessel process of dimension $3-8/\kappa$ starting from $\eps$ and stopped when it reaches zero where $\kappa=16/3$. Moreover, it is independent of $(\theta(t), t\le T_k^{\eps})$. 
\end{corollary}

\subsection{Convergence in law to a Bessel process}\label{ss.Bessel}
\label{subsec::cvg_besselprocess}
In this section, we will prove in Proposition~\ref{prop::cvg_Besselprocess} that $\theta$ considered as a whole process is indeed a Bessel process and we will complete the proof of Theorem~\ref{thm::cvg_interface_chordal}.
\begin{proposition}\label{prop::cvg_Besselprocess}
The law of $(\theta(t)/\sqrt{\kappa}, t\ge 0)$ is the same as that of a Bessel process of dimension $3-8/\kappa$ where $\kappa=16/3$. 
\end{proposition}

\begin{proof}
In order to apply Lemma~\ref{lem::approximate_bessel} to the process $\theta/\sqrt{\kappa}$, we need to check the two requirements there. The first item holds due to Corollary~\ref{cor::cvg_oneBesselexcursion}. For the second item, Theorem~\ref{thm::chordal_loewner_cvg} guarantees that $\eta$ is a continuous curve with continuous driving function. By Lemma~\ref{lem::chordal_evolution_zeroLeb}, we see $\Leb\{t: \theta(t)=0\}\le \Leb\{t: \eta(t)\in\R\}=0$
which confirms the second item.
\end{proof}

\begin{proof}[Proof of Theorem~\ref{thm::cvg_interface_chordal}]
Recall the notations at the beginning of Section~\ref{subsec::key_setup}. By Theorem~\ref{thm::rsw_strong}, the collection of interfaces $\{\eta^{\delta}\}_{\delta}$ satisfies Condition C2. 
By Theorem~\ref{thm::chordal_loewner_cvg}, the sequence is tight. Suppose $\delta_n\to 0$ and $\{\eta^{\delta_n}\}_n$ is a convergent subsequence and the limit is denoted by $\eta$. Theorem~\ref{thm::chordal_loewner_cvg} also gives that $\eta$ is a continuous curve with continuous driving function $W$. We denote by $\theta_t$ the renormalized harmonic measure of the right side of $\eta[0,t]$ union $[0,\phi(b)]$ in $\HH\setminus\eta[0,t]$ seen from infinity.
By Lemma~\ref{lem::processvsrhm}, we can apply Lemma~\ref{lem::chordal_evolution_boundary_point} to $\eta$, thus
\[\theta_t+W_t=\int_0^t\frac{2ds}{\theta_s},\quad\forall t\ge 0.\]
By Proposition~\ref{prop::cvg_Besselprocess}, we know that $\theta(t)/\sqrt{\kappa}$ is a Bessel process of dimension $3-8/\kappa$. Thus $(W_t, \theta(t)+W_t: t\ge 0)$ solves~\eqref{eqn::chordal_sle_sde}, i.e. by setting $V_t=\theta_t+W_t$, we have 
\[dW_t=\sqrt{\kappa}dB_t+\frac{(\kappa-6)dt}{W_t-V_t},\quad dV_t=\frac{2dt}{V_t-W_t}.\]
Thus $\eta$ is an $\SLE_{\kappa}(\kappa-6)$. As $\SLE_{\kappa}(\kappa-6)$ is the only subsequential limit, we have the convergence of the sequence. 
\end{proof}

\section{Detailed sketch of the convergence to radial $\SLE_{16/3}(16/3-6)$ and the one-arm exponent}
\label{sec::discussion_radial}

The goal of this section is to give a detailed sketch of the different steps needed to adapt the proof in the chordal case to the radial case. This should not be considered a complete proof, in particular in the case of item 4) below whose complete proof would require more topological details. In Subsection \ref{ss.onsager}, we sketch how to derive Onsager's exponent $1/8$ from the convergence of the radial exploration process.

We give a brief discussion on Bernoulli percolation and its one-arm exponent here. The Bernoulli percolation exploration path was introduced by Schramm and was conjectured to converge to $\SLE_6$ in~\cite{SchrammFirstSLE}. Shortly afterwards, Smirnov proved \cite{SmirnovPercolationConformalInvariance} the conformal invariance of the scaling limit of crossing probabilities for Bernoulli site percolation on triangular lattice---Cardy's formula---and gave a strategy to prove convergence to $\mathrm{SLE}_6$. A detailed proof of the convergence towards chordal $\SLE_6$ was provided by Camia-Newman in~\cite{CamiaNewmanPercolation} (see also the lecture notes of Werner \cite{WernerLecturePercolation} which provide a proof closer to the strategy given in \cite{SmirnovPercolationConformalInvariance}). As pointed in~\cite{CamiaNewmanPercolation}, the one-arm exponent of Bernoulli site percolation~\cite{LawlerSchrammWernerOneArmExponent} requires the convergence of the exploration path in the radial case, which does not come for free from the chordal case; and they addressed this gap in \cite{CamiaNewmanPercolationFull}. These technical difficulties in Bernoulli percolation also appear in the case of FK-Ising percolation. On top of them, in FK-Ising percolation, we also need to treat the convergence of renormalized harmonic measure process.

\subsection{On the convergence to radial $\SLE_{16/3}(16/3-6)$}

One possible way to obtain the convergence to radial $\SLE_{16/3}(16/3-6)$ would be to design and analyse a discrete parafermionic observable well-adapted to a radial exploration process. This would be in some sense the approach followed in \cite{KemppainenSmirnovBoundaryTouchingLoopsFKIsing, KemppainenSmirnovFullLimitFKIsing}. Once one has the convergence to chordal $\SLE_{16/3}(16/3-6)$, another natural route is to extract the convergence to radial $\SLE_{16/3}(16/3-6)$ using the fact that the chordal $\SLE_{16/3}(16/3-6)$ is \textbf{target-independent} (Lemma \ref{lem::sle_targetindep}). Even tough very natural, this strategy does not come for free and the following steps need to be addressed in order to prove the convergence to radial $\SLE_{16/3}(16/3-6)$: 

\begin{figure}[ht!]
\begin{center}
\includegraphics[width=0.6\textwidth]{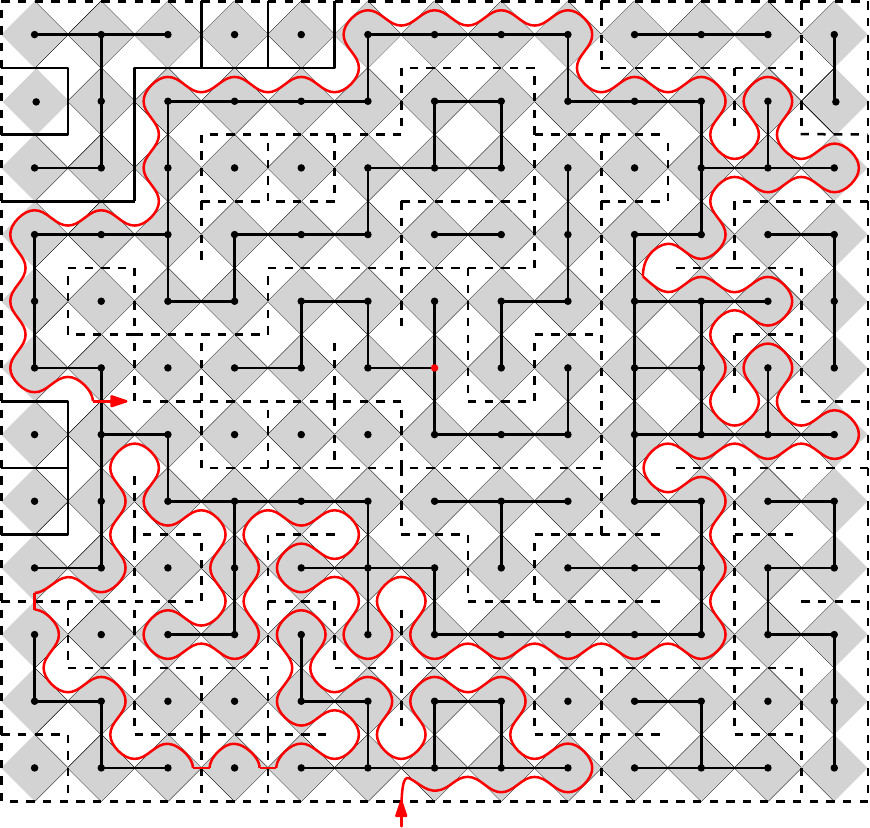}
\end{center}
\caption{\label{fig::fkinterface_radial}}
\end{figure}

\begin{enumerate}
\item First, the powerful topological setup from \cite{KemppainenSmirnovRandomCurves} needs to be adapted to radial curves. In particular, one needs to show radial analogs of statements such as Theorem~\ref{thm::chordal_loewner_cvg}. It turns out that by following closely the proof from \cite{KemppainenSmirnovRandomCurves}, there are no real difficulties on the way for this first item.  
\item Then let us consider the exploration path in the radial case. Fix a domain $\Omega$, a boundary point $a$ and an interior point $z$. Let $(\Omega^{\delta}; z^{\delta}, a^{\delta})$ be a sequence of discrete domains on $\delta\Z^2$ converging to $(\Omega; z, a)$ in the Carath\'{e}odory sense: fix the conformal map $\phi: (\Omega; z, a)\to (\U; 0, 1)$ and $\phi^{\delta}: (\Omega^{\delta}; z^{\delta}, a^{\delta})\to (\U; 0, 1)$ so that $\phi^{\delta}\to\phi$ as $\delta\to 0$ uniformly on compact subset of $\Omega$. Consider the FK-Ising model on $\Omega^{\delta}$ with full free boundary condition. Suppose $\gamma^{\delta}$ is the exploration path from $a^{\delta}$ to $z^{\delta}$ defined as follows. Starting from $a^{\delta}$, cut open the nearest loop and follow the loop counter-clockwise until one of the following three cases happens: (1) the path reaches the target; (2) the connected component containing the target in the complement of the path has fully wired boundary condition; (3)the path arrives at a point which is disconnected from the target. If case (1) or (2) happens, the path stops. If case (3) happens, cut open the loop nearest to the current position and follow the new loop counter-clockwise until one the three cases happens, and repeat the same strategy. Continue in this way until either case (1) or (2) happens, and when it happens, the path stops and we denote this time by $\tau^{\delta}_{ccw}$. Note that the exploration path continues so that it is not disconnected from the target and that primal cluster is on the left and dual cluster is on the right as long as it is possible. When it is not possible, the path continues so that primal cluster is on the left. See Fig.~\ref{fig::fkinterface_radial}.
Suppose $\delta_n\to 0$, and denote by $\LA^n=\LA^{\delta_n}$ for $\LA=\Omega, z, a, \phi, \gamma, \tau$ and denote $\eta^n=\phi^{n}(\gamma^n)$. 
Then, one proceeds exactly as in the chordal case with stopping times $T^{n,\eps}_k, S^{n,\eps}_k$ etc.

\item 
As both $\{\gamma^n\}$ and $\{\eta^n\}$ are tight, one can extract subsequence, which we still denote by $\{\gamma^n\}$ and $\{\eta^n\}$, such that $W^n$ converges in distribution to $W$ and $\eta^n$ converges in distribution to $\eta$ locally uniformly. We couple them on the same probability space so that they converge almost surely. 
Here, one needs to show that for the coupled interfaces $(\eta^n,\eta)$ one has $T_k^{n,\eps}\to T_k^{\eps}$, $S_k^{n,\eps}\to S_k^{\eps}$, and $\tau_{ccw}^n \to \tau_{ccw}$ in probability. As explained above, this step requires some care as the use of a stopping time for the continuous curve $\eta$ will ruin the spatial Markov property for $\eta^n$.
The techniques we used in the chordal case (see Proposition~\ref{pr.ST}) work in the same fashion in the radial setting except one needs to add the following step:

\item

Similarly as in Proposition \ref{pr.thetan}, we need to show in the radial setting the uniform convergence of discrete harmonic measures $\theta^n$. More precisely, 
let $\eta^n=\phi^n(\gamma^n)$ and $\eta=\phi(\gamma)$ be the radial curves conformally mapped into $\U$ and parametrised by their disk--capacity. From the analog of Theorem  \ref{thm::chordal_loewner_cvg} mentioned in item 1., we get that $\eta^n\to \eta$ and $W^n\to W$ locally uniformly. 
Let $t\mapsto \theta^n(t)$  (resp $\theta(t)$) denote the harmonic measure of the wired arc on the left of $\eta^n[0,t]$ (resp. $\eta[0,t]$). With these notations, we need to show that for any fixed $T>0$,
\begin{align}\label{}
\| \theta^n  - \theta \|_{\infty, T} = \sup_{t\in [0,T]} |\theta^n(t) - \theta(t)| \to 0\,.
\end{align}
A slightly different proof as the one we used in Section \ref{s.RHM} is needed here, as the proof in Section \ref{s.RHM} relies  specifically on the geometry of $\H$. One possible way to proceed is to divide the proof in the following two steps:
\begin{enumerate}
\item[a)] Equicontinuity of the family $\{\theta^n\}_n$. As pointed out to us by Avelio Sep\'ulveda, one can obtain the equicontinuity of $\{\theta^n\}_n$ by relying on $\eta^n \to \eta$ locally uniformly together with some harmonic measure considerations. Indeed $\eta^n \to \eta$ locally uniformly implies that for any $T>0$ and $r>0$, there exists $\delta>0$, s.t. for $n$ large enough, $\eta^n([t,t+\delta])$ remain inside the ball $B(\eta^n(t),r)$ for any $t\in[0,T]$. Together with some easy harmonic measure estimates, this implies the equicontinuity of $\{\theta^n\}_n$.  As such it reduces the question to the following second step.

\item[b)] 
Pointwise convergence of $\theta^n \to \theta$. Let us then fix some $t\geq 0$. Consider the $\delta$-neighborhood $O_n^\delta$ of $\eta^n([0,t])$. Let $F_n^\delta \subset \p O_n^\delta$ be the closed set of points on the boundary of $O_n^\delta$ which are at geodesic-distance $\delta$ measured in $\U\setminus \eta^n([0,t])$ from the wired arc of $\eta^n([0,t])$ and which are at Euclidean distance at least $\delta^{1/100}$ from the tip $\eta^n(t)$ as well as from the current force point (last disconnection vertex). We claim that with high probability (as $\delta\to0$), all points in $F_n^\delta$ are at a geodesic-distance at least $\delta^{1/2}$ measured in $\U\setminus \eta^n([0,t])$ from the free arc of $\eta^n([0,t])$: otherwise, one could find a six-arm event for the FK percolation (three-arm event if near the boundary of $\p \U$), in an annulus of inner radius $\delta^{1/2}$ and outer radius $\delta^{1/100}$. This can be shown to be of vanishing probability (vanishing in $\delta \to 0$, uniformly in $n$) using the fact that the six-arm exponent for critical FK-Ising percolation is $>2$. See for example \cite[Lemma~6.1]{CamiaNewmanPercolationFull}. Now, similarly as in Section \ref{s.RHM}, one can use Beurling's estimate to claim that once a Brownian motion in $\U \setminus O_n^\delta$ started at 0 will hit the set $F_n^\delta$, it will hit with very high probability the wired arcs of $\eta^n([0,t])$ as well as $\eta([0,t])$ before intersecting the respective free arcs. One concludes by some easy harmonic measure considerations for the balls of radius $\delta^{1/100}$ around the tip and the force point. More topological details are certainly needed to turn this sketch into a formal proof.

\end{enumerate}

\item As in the chordal case, one needs to justify limits of conditional expectations arising after stopping times. (I.e. the fact the spatial Markov property passes to the scaling limit definitively requires some justification). This step can be handled similarly as in Subsection~\ref{subsec::cvg_oneBesselexcursion}.

\item Finally, one can extract the radial Loewner driving function $W$ from the angle $\theta$ evolving like a $\mathrm{cotan}$-Bessel process on $[0,2\pi]$ by relying on a suitable radial analog of Lemma~\ref{lem::chordal_evolution_boundary_point}.  It is not so immediate to generalize Lemma~\ref{lem::chordal_evolution_boundary_point} to the radial setting, because the assumption $\Leb(\eta\cap\R)=0$ does not suit the radial setting. One possible way is to compromise to an almost sure conclusion (instead of deterministic conclusion): in the chordal setting, one replaces the requirement $\Leb(\eta\cap\R)=0$ by Condition C2, since $\Leb(\eta\cap\R)=0$ holds almost surely under Condition C2 (see Lemma~\ref{lem::processvsrhm}), then the conclusion holds almost surely. This compromised version of Lemma~\ref{lem::chordal_evolution_boundary_point} in the chordal setting is easily generalized to the radial setting.

\end{enumerate}

\subsection{On Onsager's one-arm exponent (equal to $1/8$)}\label{ss.onsager}

We denote by $\mu_{\Lambda_n, \beta_c}^{\oplus}$ the probability measure of the spin-Ising model on $\Lambda_n=[-n,n]^2\subset\Z^2$ with $\oplus$ boundary condition and the critical inverse temperature $\beta_c$. In this section, we will discuss the decay of the expectation of the spin at the origin $\sigma_0$:  as $n\to\infty$,
\begin{equation}\label{eqn::onsager}
\mu_{\Lambda_n, \beta_c}^{\oplus}[\sigma_0]\asymp n^{-\frac{1}{8}}.
\end{equation}
The exponent $1/8$ is one half of $1/4$ which is the power in the decay of the two-point correlation function in the Ising model. 
The history of two-point correlation function dates back to Onsager in 1940s. 
In \cite{McCoyWu}, McCoy and Wu computed many important quantities of the Ising model including the critical exponent $1/4$ in two-point correlation function. 
In \cite{ChelkakHonglerLzyurovConformalInvarianceCorrelationIsing}, the authors derived~\eqref{eqn::onsager} using holomorphic spinor observables and obtained $\lim_n n^{\frac{1}{8}}\mu_{\Lambda_n, \beta_c}^{\oplus}[\sigma_0]$ and its generalizations. See also the useful reference \cite{ChelkakECM}.

By Edwards-Sokal coupling, the expectation in~\eqref{eqn::onsager} is related to the crossing probability in FK-Ising model: 
\[\mu_{\Lambda_n, \beta_c}^{\oplus}[\sigma_0]=\FK{\Lambda_n, p_c(2)}{\mathrm{wired}}{0 \leftrightarrow \p \Lambda_n}. \]
In this section, we briefly outline another derivation of 1/8 assuming the convergence of FK-Ising interface to radial $\SLE_{16/3}(16/3-6)$ (i.e. assuming item 4 above).

First let us point out that the convergence to radial $\SLE_{16/3}(16/3-6)$ only implies a weaker result than Onsager's one: similarly to the one-arm exponent for critical percolation \cite{LawlerSchrammWernerOneArmExponent}, it implies that as $n\to \infty$, 
\begin{align*}\label{}
\FK{\Lambda_n, p_c(2)}{\mathrm{wired}}{0 \leftrightarrow \p \Lambda_n} = n^{-\frac 1 8 +o(1)}\,,
\end{align*}
while there is no sub-polynomial $o(1)$ correction in Onsager's result. The main steps to prove this are as follows:
\begin{enumerate}
\item A computation of the exponent $1/8$ on the continuous level. This corresponds to the one-arm exponent $\tilde{\alpha}_1$ of radial $\SLE_{\kappa}(\kappa-6)$ which was calculated in \cite{SchrammSheffieldWilsonConformalRadii}: 
\[\tilde{\alpha}_1=\frac{(8-\kappa)(3\kappa-8)}{32\kappa}.\]
Note that $\tilde{\alpha}_1=1/8$ when $\kappa=16/3$. 

\item Second, one needs to carefully argue that this continuous one-arm exponent $\tilde{\alpha}_1$ is the same as the exponent describing the crossing probability, for critical FK-Ising percolation, of large macroscopic annuli $\Lambda_n \setminus \Lambda_{rn}$ as $n\to \infty$ and $r\in (0,1)$. This step can be made rigorous in essentially two steps: a) first by showing similarly as in Proposition~\ref{pr.ST} that the discrete disconnection times for the radial exploration process converge to the continuous ones. 
And b) by showing via some separation types of lemmas that the probability of not disconnecting (i.e. $\theta$ not reaching 0) is up to constant the same as connecting $\p \Lambda_n$ to $\p \Lambda_{rn}$. 

\item Finally as for critical percolation ($q=1$), one relies on the \textbf{quasi-multiplicativity} of the discrete one-arm event to conclude (see \cite{LawlerSchrammWernerOneArmExponent} or \cite[Section~4.2]{SmirnovWernerCriticalExponents} in the case $q=1$).
In fact this quasi-multiplicativity of the one-arm event is even rigorously known for all critical random-cluster models with $q\in [1,4]$ thanks to the recent Russo-Seymour-Welsh Theorem proved in \cite[Theorem~7]{DuminilSidoraviciusTassionContinuityPhaseTransition}.

\end{enumerate}

We end this article with a discussion on how to apply our approach to other lattice models. The goal of this approach is to pass the convergence of interface with Dobrushin boundary condition to the convergence with other boundary conditions in local uniform topology on curves (as in Section~\ref{subsec::cvg_curves_chordal}). The approach requires two main ingredients: 
\begin{enumerate}
\item[(a)] the convergence of interface with Dobrushin boundary condition to $\SLE_{\kappa}$;
\item[(b)] the RSW theorem in the discrete model which enables us to apply the topological tool developed in~\cite{KemppainenSmirnovRandomCurves}. 
\end{enumerate}
So far, item (a) is known for the following models: loop-erased random walk (LERW) and uniform spanning tree (UST) \cite{LawlerSchrammWernerLERWUST}; Ising and FK-Ising \cite{CDCHKSConvergenceIsingSLE}; level lines in discrete Gaussian free field (GFF) \cite{SchrammSheffieldDiscreteGFF} and percolation \cite{CamiaNewmanPercolation}. Among these models, item (b) is known for LERW, Ising and FK-Ising, and percolation, and it fails for UST, see~\cite[Section~4]{KemppainenSmirnovRandomCurves}. For level lines in discrete GFF, the RSW crossing estimate is believed to hold, especially thanks to the recent works by Titus Lupu (see for example \cite{lupu})  yet it is not written anywhere to our knowledge. For instance, the convergence of the level line with Dobrushin boundary condition was proved in \cite{SchrammSheffieldDiscreteGFF}; however, in the same article, the authors derived the convergence of the level line with zero boundary condition in the sense of driving function, not in the sense of the convergence in curves, precisely due to the missing of such RSW estimates.

\newcommand{\etalchar}[1]{$^{#1}$}

\end{document}